\documentclass[10pt]{amsart}
\usepackage[margin=1.03in]{geometry}
\usepackage{amsmath,amsfonts,amssymb,amsthm,epsfig,color, tikz}
\usepackage{graphicx} 
\usepackage{scalerel,stackengine}
\stackMath
\newcommand\reallywidehat[1]{%
\savestack{\tmpbox}{\stretchto{%
  \scaleto{%
    \scalerel*[\widthof{\ensuremath{#1}}]{\kern.1pt\mathchar"0362\kern.1pt}%
    {\rule{0ex}{\textheight}}
  }{\textheight}%
}{2.4ex}}%
\stackon[-6.9pt]{#1}{\tmpbox}%
}

\newcommand{\tr}{\operatorname{tr}}
\newcommand{\bb}{\mathbb}
\newcommand{\C}{\bb C}

\newcommand{\SV}{\operatorname{SV}}
\newcommand{\Z}{\bb Z}
\newcommand{\R}{\bb R}

\newcommand{\abs}[1]{\left|{#1}\right|}

\usepackage{float}
\usepackage{hyperref}
\newcommand{\one}{\mathbf{1}}
\newcommand{\re}{\operatorname{Re}}
\newcommand{\im}{\operatorname{Im}}

\newcommand{\T}{\mathcal T}
\newcommand{\hh}{\mathcal H}
\newcommand{\om}{\omega}

\newcommand{\La}{\Lambda}
\newcommand{\SC}{\operatorname{SC}}
\newcommand{\minuszero}{\backslash\{0\}}

\newcommand{\norm}[1]{\left\lVert{#1}\right\rVert}

\newtheorem{Theorem}{Theorem}
\newtheorem{Cor}[Theorem]{Corollary}
\newtheorem{Prop}[Theorem]{Proposition}
\newtheorem{lemma}[Theorem]{Lemma}

\newtheorem{rmk}[Theorem]{Remark}

\numberwithin{equation}{section}
\numberwithin{Theorem}{section}

\begin{document}

\title{Counting pairs of saddle connections}
\author{J.~S.~Athreya}
\author{S.~Fairchild}
\author{H.~Masur}
\begin{abstract}
	We show that for almost every translation surface the number of pairs of saddle connections with bounded magnitude of the cross product has asymptotic growth like $c R^2$ where the constant $c$ depends only on the area and the connected component of the stratum. The proof techniques combine classical results for counting saddle connections with the crucial result that the Siegel--Veech transform is in $L^2$.  In order to capture information about pairs of saddle connections, we consider pairs with bounded magnitude of the cross product since the set of such pairs can be approximated by a fibered set which is equivariant under geodesic flow. In the case of lattice surfaces, small bounded magnitude of the cross product is equivalent to counting parallel pairs of saddle connections, which also have a quadratic growth of $c R^2$ where $c$ depends in this case on the given  lattice surface.
\end{abstract}
\maketitle

\section{Introduction}

\paragraph*{\bf Translation surfaces and saddle connections} A translation surface $(X, \omega)$ is a pair consisting of a compact Riemann surface $X$ and $\omega$, a non-zero holomorphic one-form. For succinctness we denote a translation surface by $\omega$ where the underlying Riemann surface is understood. A \emph{saddle connection} on $\omega$ is an oriented geodesic in the flat metric determined by $\omega$ connecting two zeros of $\omega$ with no zeros in its interior. Let $\SC_{\omega}$ be the set of saddle connections on $\omega$. For $\gamma \in \SC_{\om}$, the associated \emph{holonomy vector} is given by $$z_{\gamma} = \int_{\gamma} \om \in \C.$$ Let $$\La_{\om} = \{z_{\gamma}: \gamma \in \SC_{\om}\}$$ denote the set of holonomy vectors of saddle connections on $\om$. This is a countable discrete subset of the plane $\C$. The \emph{length} $\ell(\gamma)$ of a saddle connection $\gamma$ is  $$\ell(\gamma) = |z_{\gamma}|.$$ For $R>0$, let $\La_\om (R) =  \La_{\om} \cap B(0, R)$ be the collection of holonomy vectors of saddle connections with length at most $R$.  

\paragraph*{\bf Pairs of saddle connections} We are interested in the distribution of \emph{pairs} of saddle connections, in particular the growth rate of the count of pairs of bounded \emph{virtual area}. The virtual area of a pair of saddle connections $\gamma, \eta$ is the given by the area $|z \wedge w|$ of the parallelogram spanned by the associated holonomy vectors $z= z_{\gamma}, w=  z_{\eta}$, where we recall that if $z= x+iy, w =u+iv \in \C$ the area $|z \wedge w|$ is given by $$|z \wedge w |= |xv - yu| = |\mbox{Im}(\bar{z} w)|.$$
We use the term \emph{virtual area} since the saddle connections whose holonomy vectors we are considering do not necessarily form an embedded parallelogram on the surface $\omega$.
Fix $A > 0$, and define the counting function $N_A(\om, R)$ to be the number of pairs of saddle connections of length at most $R$ and whose virtual area is at most $A$
 $$N_{A}(\omega, R) =\#\{(z, w) \in \La_{\om}(R)^2:  |z \wedge w| \le A, |w|\leq |z|\}.$$ 

\paragraph*{\bf Moduli spaces}   The moduli space $\Omega_g$ of compact genus $g$ area $1$ translation surfaces (where $(X_1, \omega_1) \sim (X_2, \omega_2)$ if there is a biholomorphism $f: X_1 \rightarrow X_2$ with $f_* \omega_2 = \omega_1$) is stratified by integer partitions of $2g-2$ (fixing the orders of the zeros of $\omega$).  The area of a surface $\om$ is given by $$\mbox{Area}(\omega) = \frac{i}{2} \int_{X} \omega \wedge \bar{\omega}.$$These strata have at most $3$ connected components~\cite{KontsevichZorich}, and each connected component $\hh$ carries a natural Lebesgue probability measure $\mu = \mu_{\hh}$~\cite{Masur82, Veech82, MasurSmillie91}. We fix $\hh$ to be a connected component of a stratum. 
   
\paragraph*{\bf Main theorem}   Our main result, motivated by considerations in geometric probability (see \S\ref{sec:poisson}) is an almost sure asymptotic growth result for the set of pairs of saddle connections with bounded virtual area.

\begin{Theorem}\label{theorem:virtual:ae} There is a constant $c_A = c_A(\mu)$ such that for $\mu$-almost every $\omega \in \hh$ $$\lim_{R \rightarrow \infty} \frac{N_A(\omega, R)}{R^2} = c_A.$$

\end{Theorem}

\subsection{History and prior results} The study of counting problems for saddle connections is very active, and connected to many different areas of mathematics, from low-dimensional dynamical systems to algebraic geometry. Motivated by problems in counting special trajectories for billiards in rational polygons, Masur~\cite{Masur90QuadraticGrowth} proved that the counting function $$N(\omega, R) = \#\La_{\om}(R) $$ has quadratic upper and lower bounds for all $\omega$, that is, there are $0 < c_1= c_1(\omega) < c_2 = c_2(\omega)$ so that for all $R$, 
\begin{equation}\label{eq:quadbounds} 
c_1 R^2  \le N(\omega, R) \le c_2 R^2 .
\end{equation}
Subsequently, Veech~\cite{Veech98} showed there is a constant $c = c(\hh)$ such that $$\lim_{R \rightarrow \infty} \int_{\hh} \left| \frac{N(\omega, R)}{R^2} - c\right|  \,d\mu(\omega)= 0,$$ an $L^1$-quadratic asymptotic result. Inspired by Veech's approach, Eskin-Masur~\cite{EskinMasur01} adapted ideas from homogeneous dynamics (specifically, the work of Eskin-Margulis-Mozes~\cite{EMM95, EMM98} on quantitative versions of Oppenheim's conjecture) and an ergodic theorem of Nevo~\cite{nevo2017equidistribution} to improve this to a pointwise asymptotic result, showing that for $\mu$-almost every $\omega \in \hh$, 
\begin{equation} \label{eq:EM01}
\lim_{R \rightarrow \infty} \frac{N(\omega, R)}{R^2} = c.
\end{equation} 

\paragraph*{\bf Error terms} More recently, Nevo-R\"uhr-Weiss~\cite{NRW20}, using error term estimates in Nevo's ergodic theorem coming from mixing properties of the \emph{Teichm\"uller geodesic flow} $g_t$ (defined in \S~\ref{sec:strategy}), showed that there is an $\alpha < 2$ such that for almost every $\omega \in \hh$, $$N(\omega, R) = c  R^2 + o(R^{\alpha}).$$    

 \subsection{The Siegel--Veech transform}\label{sec:siegelveech} A crucial ingredient in the work of Veech~\cite{Veech98} is the \emph{Siegel--Veech transform}. Let $B_c(X)$ be the space of bounded measurable functions with compact support on a space $X$.  For $f \in B_c(\C),$ we define a function $\widehat{f}$ on $\hh$ by $$\widehat{f}(\omega) = \sum_{z \in \La_{\om}} f(z).$$ For example, if $f = \one_{B(0, R)}$ is the indicator function of $B(0, R)$, $$\widehat{f}(\omega) = N(\omega, R).$$ 
 
\paragraph*{\bf The Siegel--Veech formula} A beautiful result of Veech~\cite{Veech98} is the \emph{Siegel--Veech formula}, which states that there is a $c = c_{SV}$ so that if $f \in B_c(\R)$ then $\widehat{f} \in L^{1}(\hh, \mu)$ and 
 \begin{equation}\label{eq:SV}
 	\int_{\hh} \widehat{f} d\mu = c \int_{\C} f(z)dz.
 \end{equation} 
In fact a crucial ingredient in Eskin--Masur's asymptotic result is that $\widehat{f} \in L^{1+\beta}$ for some $\beta >0$. We will need similar results for a generalized Siegel--Veech transform. Given $h  \in B_c(\C^2)$, we define a function $\widehat{h}$ on $\hh$ by 
$$\widehat{h}(\omega) = \sum_{z_1, z_2 \in \La_{\om}} h(z_1, z_2).$$ For example, if $h = \one_{D_A(R)}$ is the indicator function of the set $$D_A(R) = \{(z, w) \in \C^2:  |w| \leq |z| \le R,  |z \wedge w| \leq A\},$$  then $$\widehat{h}(\omega) = N_A(\omega, R).$$ 

\paragraph*{\bf Higher integrability} In our proof of Theorem~\ref{theorem:virtual:ae}, we rely on a result of Athreya-Cheung-Masur~\cite{AthreyaCheungMasur19} which shows that $\widehat{h} \in L^{1+\beta}$ for $h \in B_c(\C^2)$ (which is equivalent to showing that if $f \in B_c(\C)$ then  $\widehat{f} \in L^{2+\beta}(\hh)$ \cite[\S 1.4.1]{AthreyaCheungMasur19}).

\subsection{The $SL(2, \R)$-action on strata.}\label{sec:sl2:strata} There is an action of the group $SL(2, \R)$ on strata. A translation surface $\omega$ gives an atlas of charts from $X\backslash \{ \omega^{-1}(0)\}$ to $\C$ whose transition maps are translations: the atlas around a point $p_0$ is given by $$z(p) = \int_{p_0}^p \omega.$$ In these coordinates, $\omega = dz$. Equivalently, such an atlas of charts determines a pair $(X, \omega)$. The group $GL^+(2, \mathbb R)$ acts by $\R$-linear-postcomposition with charts, and the group $SL(2, \R)$ preserves the set of surfaces with area $1$. The measure $\mu_{\hh}$  is ergodic and invariant under the $SL(2, \R)$-action, and is locally given by Lebesgue measure in appropriate coordinates on $\hh$. Note that the assignment \begin{equation}\label{eq:assignment}\omega \mapsto \La_{\omega}\end{equation} is $SL(2, \R)$-equivariant, that is $$\La_{g\omega} = g\La_{\omega}.$$

\subsection{Random subsets of $\C$}\label{sec:context} Selecting $\om$ at random according to $\mu$, the assignment (\ref{eq:assignment}) gives a notion of a random discrete subset of $\C$. The equivariance means that this choice of random discrete subset is $SL(2,\R)$-invariant. More formally, using the assignment (\ref{eq:assignment}) allows us to push the measure $\mu$ forward to a measure on the space $\mathcal D(\C)$ of discrete subsets of $\C$, and the equivariance implies that this measure is $SL(2, \R)$-invariant where the action on $\mathcal D(\C)$ is induced by the action on $\C$. This perspective was originated by Veech~\cite{Veech98}.

\subsubsection{Poisson processes}\label{sec:poisson} Perhaps the most famous example of a (family of) $SL(2,\R)$- probability invariant measures on  $\mathcal D(\C)$ are \emph{planar Poisson processes}, see, for example, Doob~\cite[\S 8.5]{Doob} for an introduction.  Fix a positive real parameter $\rho$ (known as the \emph{intensity}). The planar Poisson process of intensity $\rho$ is a probability measure $\mathbb P_{\rho}$ on $\mathcal D(\C)$ characterized by the following property~\cite{Moyal}: given any finite collection of finite measure Borel subsets $A_1, \ldots A_k \subset \C$ with $a_i = m(A_i) < \infty$, where $m$ denotes Lebesgue measure on $\C$, the random variables $X_i$ recording the number of points in $A_i$ are independent Poisson random variables with mean $\rho a_i$, $i=1, \ldots, k$. That is, $$\mathbb P_{\rho} \left(\La \in \mathcal D(\C): X_i (\La) = \#(\La \cap A_i) = n\right)) = e^{-\rho a_i} \frac{(\rho a_i)^n}{n!}.$$ In particular, these random variables have all moments (that is, $E(X_i^p) = \int_{\mathcal D(\C)} X_i(\La) d\mathbb P_{\rho}(\La) < \infty$ for all $p >0$).

\paragraph*{\bf Comparison to Poisson processes} Much of the study of the geometry and statistics of the sets $\La_{\om}$ has been in some sense motivated by the comparison of these statistics to those of Poisson processes. For example, the Siegel--Veech formula (\ref{eq:SV}) shows that, at least at the level of the mean, the random sets $\La_{\om}$ behave like planar Poisson processes of intensity $c_{SV}$. In terms of counting \eqref{eq:EM01} also gives evidence that $\La_{\om}$ behaves like planar Poisson processes as follows. One can show that with probability $1$ (see, for example~\cite[Lemma 2.1]{Miles}) that the number of points of a planar Poisson process of intensity $\rho$ in $B(0,R)$ grows like $\rho \pi R^2$ as $R \rightarrow \infty$, that is, letting $N(\La, R) = \# \La \cap B(0, R)$ for $\La \in \mathcal D(\C)$, then for $\mathbb P_{\rho}$-almost every $\La \in \mathcal D(\C)$, $$\lim_{R \rightarrow \infty} \frac{N(\La, R)}{\pi R^2} = \rho.$$ 

\paragraph*{\bf Variance and higher moments} The $L^2$-estimates of~\cite{AthreyaCheungMasur19} can be viewed as showing that the random variables corresponding to counting the intersection of finite measure Borel sets with $\La_{\om}$ have finite \emph{variance}, though as observed there, they do not have finite third moments, marking a difference from Poisson processes. 

\paragraph*{\bf Counting pairs for Poisson processes} If one considers our problem for planar Poisson processes of intensity $\rho$, it is an exercise in geometric probability to show that with $\mathbb P_{\rho}$-probability $1$, that $$N_A(\La, R) = \#\{(z, w) \in \La \times \La: |w| \le |z| \le R, |z \wedge w| \le A\},$$ will grow like the $\rho^2$ times the volume of the region $\{(z, w) \in \C^2: |w| \le |z| \le R, |z \wedge w| \le A\},$ which a direct computation shows grows like a constant (depending linearly on $A$) times $R^2$. So at the level of these asymptotics, there is a strong resemblance between the set of saddle connections of a random translation surface, and a planar Poisson process.

\paragraph*{\bf Directional statistics} There has also been a wide range of recent work, starting with~\cite{AChaika} and further developed in~\cite{FChaika} on understanding the fine-scale distribution of gaps between \emph{directions} of saddle connections and comparing them to that of Poisson processes. The current work can be seen as a way of connecting counting problems and the idea of distribution of directions: if a pair of saddle connections are long and have bounded virtual area, their directions are quite close to each other.

\subsubsection{The constants $c_A$}\label{sec:cA} Motivated by the comparison to Poisson processes, it is an interesting question to understand the constants $c_A$ as as a function of $A$, which would be linear in $A$ if it mirrors the Poisson process. By definition, $c_A$ is montonically nondecreasing in $A$. It seems reasonable to conjecture that in fact $c_A$ should go to $\infty$ as $A\to\infty$. We are able to show the simple result pointed out to the authors by Alex Wright that at least $c_A >0$.

\begin{Cor}[Corollary of quadratic growth]\label{cor:cApositive}
For $A\geq 1$  we have that $c_A>0$.
\end{Cor}
\begin{proof}
Given $\omega$ there  are constants $c>0$ and  $R_0$ so that for $R\geq R_0$ the number of cylinders on $\omega$ with circumference at most $R$ is at least $cR^2$ \cite{Ma3}.
Let $\beta$ be a saddle connection on the boundary of the cylinder, and let $\gamma$ be a saddle connection that crosses the cylinder. Since we are working with unit area surfaces, the area of the cylinder is at most $1$, and so we have $|z_{\beta}\wedge z_{\gamma}|\leq 1$. This gives a pair of saddle connections of length at most $R$ and so $$N_A(\omega,R)\geq cR^2.$$
This proves that $c_A\geq c>0$. 
\end{proof}

\paragraph*{\bf Computing $c_A$} An interesting question is whether $c_A$ is computable. In the case of a single saddle connection, work of Eskin--Masur--Zorich \cite{EMZ} shows that the Siegel--Veech constant $c(\hh)$ can be computed in terms of volumes of strata of translation surfaces. It is still an open question whether an analogous result is true for the current problem. Indeed in~\cite{Veech98}, the fact that the only nonatomic ergodic $SL(2,\R)$ invariant measure on $\C$ is a multiple of Lebesgue measure is used in a crucial way to prove the Siegel--Veech formula (\ref{eq:SV}). What makes the computation of $c_A$ in our setting interesting and difficult is that we do not yet explicitly know how to compute integrals of Siegel--Veech transforms of functions on $\C^2$, see \S\ref{sec:svmeasures} for more details.

\subsubsection{Other measures}\label{sec:othermeasure}
 Theorem~\ref{theorem:virtual:ae} does not immediately generalize to almost-everywhere results for arbitrary ergodic $SL(2,\R)$ invariant measures $\nu$. In order to guarantee first that $N_A(\omega, R)$ is integrable, we need to know that the function counting individual saddle connections is in $L^2(\nu)$ (i.e., the counting function has finite variance), which was only known for the measure $\mu$ discussed above (and some other families of measures) by \cite{AthreyaCheungMasur19}. 

\subsection{Strategy of proof}\label{sec:strategy} 
Our approach will expand on ideas from Eskin--Masur~\cite{EskinMasur01} in significant ways.  We set up our counting problem as an integral of a Siegel--Veech transform over a piece of an $SL(2, \mathbb R)$-orbit on $\hh$, and then apply the Nevo ergodic theorem~\cite[Theorem 1.1]{nevo2017equidistribution}. Since we are counting pairs of saddle connections, the construction of the function whose Siegel--Veech transform we integrate is considerably more intricate, requiring a carefully chosen domain in $\C^2$ and some technical, but elementary arguments given in Section~\ref{sec:estimates}. In addition we also need measure approximations for subsets of the stratum, which also require new ideas beyond the estimates of~\cite{AthreyaCheungMasur19}. In particular in order to get bounds on counting pairs in the thin part of the stratum $\hh$, we define three families of sets which exhaust  the thin part of the stratum depending on the shortest saddle connection and its relation to another non-homologous saddle connection (Section~\ref{sec:proof4.2}).   We also need to develop measure bounds for pairs in the thick part satisfying certain closeness conditions (Section~\ref{sec:subquadraticthick}).

\paragraph*{\bf Renormalization} We now outline in more detail the strategy of proof of Theorem~\ref{theorem:virtual:ae}, emphasizing the similarities to previous work. At the heart of the argument is that a saddle connection of holonomy $r e^{\pi/2-i\theta}$ on $\omega$ becomes a saddle connection of holonomy $re^{-t} i$ on $g_tr_\theta \omega$, so if the length $r$ of the original saddle connection has $|r| < e^t$, the corresponding holonomy on $g_t r_{\theta} \om$ is vertical and has length at most $1$. We recall the strategy of Eskin--Masur~\cite{EskinMasur01} for understanding the counting function $N(\omega, R)$: they construct a function $f \in B_c(\C)$ (essentially the indicator function of a trapezoid, see Figure~\ref{fig:arcs}), which satisfies $$\frac{1}{2\pi} \int_0^{2\pi} f(g_t r_{\theta} z) \,d\theta\approx e^{-2t} \one_{A\left(\frac{e^t}{2}, e^t\right)}(z),$$ where the matrices
\begin{equation}\label{eq:gtrtheta}
g_t = \begin{pmatrix} e^t & 0 \\ 0 & e^{-t}\end{pmatrix} \quad r_\theta = \begin{pmatrix}
	\cos \theta & -\sin \theta \\ \sin \theta & \cos\theta
\end{pmatrix}
\end{equation}
 act $\R$-linearly on $\C$, and for $0 < R_1 < R_2$, $$A(R_1, R_2) = \{ z \in \C: R_1 < |z| < R_2\}.$$ Putting $e^t = R$, the Siegel--Veech transform $\hat{f}(g_tr_\theta\omega) = \sum_{z\in \Lambda_\omega} f(g_tr_\theta z)$ adds the above expression over all $z \in \La_{\omega}$, and transforms the integration from over $\C$ to integrating over the stratum to obtain $$\frac{1}{2\pi} \int_{0}^{2\pi} \widehat{f}(g_t r_{\theta} \omega) \, d\theta \approx \frac{1}{R^2} \left( N(\omega, R) - N(\omega, R/2) \right) .$$ This reduces the counting problem to a problem of understanding the sequence of integrals $$\frac{1}{2\pi} \int_{0}^{2\pi} \widehat{f}(g_t r_{\theta} \omega) \, d\theta.$$ 
 Nevo's ergodic theorem (Theorem~\ref{thm:Nevo}) deals precisely with integrals of this form, but with some compactness and smoothness assumptions on the integrand. Theorem~\ref{thm:Nevo} gives that almost surely the integrals converge to $\int\widehat{f}d\mu$. The Siegel--Veech formula is then applied to say that this last integral is $c \int_{\mathbb C} f(z) dz$.  
 
 
\subsubsection{Pairs} In our case we will construct as the main part of the proof,  a function $h_A \in B_c(\C^2)$ so that 
\begin{equation}\label{eq:circle} \frac{1}{2\pi} \int_0^{2\pi} h_A(g_t r_{\theta} (z, w)) \approx \frac{e^{-2t}}{\pi} \one_{D_A(e^t/2, e^{t})}(z, w),\end{equation} where for $R_1 < R_2$, $$D_A(R_1, R_2) = \{(z, w) \in \C^2: |z \wedge w| \le A, |w| \le |z|, R_1 \le |z| \le R_2\}$$ and the action of $g\in SL(2, \R)$ on $(z,w) \in\C^2$ is the diagonal $\R$-linear action $g\cdot(z,w) = (gz,gw)$. Adding (\ref{eq:circle}) over all $(z, w) \in \La_{\om}^2$, we will prove for large $R$
\begin{equation*} 
\label{eq:limits} \frac{1}{2\pi} \int_{0}^{2\pi} \widehat{h_A}(g_{\log(R)} r_{\theta} \omega)\, d\theta \approx \frac{N_A(\omega, R) - N_A(\omega, R/2)}{\pi R^2}.\end{equation*} 
\paragraph*{\bf Circle averages} Once again we will need to show that the limit of the circle averages $$\lim_{t\to\infty}\frac{1}{2\pi} \int_{0}^{2\pi} \widehat{h_A}(g_t r_{\theta} \omega)\, d\theta$$ exists. To do that we again will implement Nevo's theorem (Theorem~\ref{thm:Nevo}) along with careful analysis of the boundary of the support of $h_A$. We will also rely on~\cite[Theorems 1.2 and 3.4]{AthreyaCheungMasur19} which shows that there is a $\kappa>0$ such that $$\widehat{h} \in L^{1+\kappa}(\hh, \mu),$$ and provides a version of the Siegel--Veech formula which works for functions defined on  $\mathbb{C}^2$.

\subsection{Organization of the paper}\label{sec:organization}
We start the paper in \S\ref{sec:parallel} by discussing the important special case of \emph{lattice surfaces}, which provides a first step in thinking about pairs of saddle connections, using fundamental results of Veech~\cite{Veech98}. In \S\ref{sec:spectral}, we state Nevo's ergodic theorem and a version of a Siegel--Veech type formula from~\cite{AthreyaCheungMasur19}, which we use to prove convergence of circle averages of Siegel--Veech transforms for continuous functions. This is Proposition~\ref{prop:convergescont}.   In \S\ref{sec:approx}, we construct our function $h_A$.  We state  the main results about $h_A$ which are  Theorem~\ref{thm:error} and Proposition~\ref{prop:quadratic} about almost sure quadratic upper bounds, and Proposition~\ref{prop:hAlimit} about convergence of circle averages for $h_A$. They are proven in the last two sections.   In  \S\ref{sec:approx},  they are used to prove Theorem~\ref{theorem:virtual:ae}. In \S\ref{sec:estimates} we bound the left side of Theorem~\ref{thm:error}  in terms of certain error terms. This is given in  (\ref{eq:decomp}).  In \S\ref{sec:upperbound}, we prove Proposition~\ref{prop:quadratic} and a lemma necessary for Proposition~\ref{prop:hAlimit}, along with a key modification to prove Proposition~\ref{prop:Mjbounds}.  These altogether give the   bounds on the error terms on the right side of (\ref{eq:decomp}) that finally allows  us to prove  Theorem~\ref{thm:error}.

\medskip

\paragraph*{\bf Acknowledgements.} We thank the Mathematical Sciences Research Institute (MSRI) where a large portion of this work was done in the Fall 2019 program on Holomorphic Differentials in Mathematics and Physics. We thank the Fields Institute where we had preliminary discussions in Fall 2018 during the program on Teichm\"uller Theory and its Connections to Geometry, Topology and Dynamics.The authors are grateful to Alex Wright for pointing out the proof of Corollary~\ref{cor:cApositive}. We also thank Jon Chaika for his suggested method to fix the proof of quadratic upper bounds. We also thank David Aulicino, Claire Burrin, Max Goering, Kanishka Katipearachchi, Yair Minsky, and John Smillie for useful discussions.  J.S.A. was partially supported by NSF CAREER grant DMS 1559860 and NSF grant DMS 2003528; the Pacific Institute for the Mathematical Sciences; the Royalty Research Fund and the Victor Klee fund at the University of Washington; and this work was concluded during his term as the Chaire Jean Morlet at the Centre International de Recherches Mathematique-Luminy. S.F. was partially supported by the Deutsche Forschungsgemeinschaft (DFG) -- Projektnummer 445466444 and 507303619.

\section{Lattice surfaces}\label{sec:parallel}

\subsection{Lattice surfaces}\label{sec:lattice} Given a surface $\om \in \hh$, we define its \emph{Veech group} $SL(\om)$ to be its stabilizer under the $SL(2, \mathbb R)$ action. A class of surfaces where counting problems are well-understood are \emph{lattice surfaces}, surfaces $\om$ whose stabilizer $\Gamma = SL(\omega)$ under the $SL(2, \R)$-action is a lattice. These are also known as \emph{Veech surfaces}. While lattice surfaces are rare, in the sense they form a set of measure $0$ in each stratum, they are a \emph{dense} set in each stratum. See Smillie-Weiss~\cite{SmillieWeiss10} and the references within for more details.

\subsubsection{Counting and orbits}\label{sec:orbits} Veech~\cite{Veech98} showed that in this setting the set of holonomy vectors $\La_{\om}$ is a finite union of orbits of the Veech group. That is, there is a finite collection of complex numbers $z_1, z_2, \ldots z_m$ such that \begin{equation}\label{eq:Veech}\La_{\om} = \bigcup_{i=1}^m \Gamma z_i. \end{equation} Using this, and techniques from homogeneous dynamics, he proved, for each $i$, there is a $c_i$ so that $$\# \left(\Gamma z_i \cap B(0, R)\right) \sim c_i R^2,$$ and thus overall quadratic asymptotics for $N(\omega, R)$.

\subsubsection{No small triangles}\label{sec:nst} Subsequently, Smillie-Weiss~\cite{SmillieWeiss10} gave many equivalent characterizations of lattice surfaces. In particular, they showed that $\om$ is a lattice surface if and only if it satisfies the \emph{no small virtual triangles (NSVT)} condition: there is an $A_0 >0$ so that for any non-parallel $z, w \in \La_{\om}$, $$|z \wedge w|>A_0.$$ So for $A <A_0$, the problem of understanding $N_A(\omega, R)$ becomes the problem of counting parallel pairs of vectors in $\La_{\om}$. 

\subsubsection{Counting parallel pairs}\label{sec:countingparallel} We write $$N_0(\omega, R) =\# \{ (z, w): |z \wedge w| = 0, |z| \le  |w| \le R \}.$$

\begin{Prop}\label{theorem:parallel} Let $\omega$ be a lattice surface. There is a constant $c = c(\omega)$ such that $$\lim_{R \rightarrow \infty} \frac{N_0(\omega, R)}{R^2} = c.$$

\end{Prop}

\paragraph*{\bf Cusps and computing $c$} We prove this result in \S\ref{sec:parallel}, and show how to compute $c$ using the decomposition of $\La_{\om}$ into orbits of $SL(X,\om)$, the structure of the cusps of the Fuchsian group $\Gamma = SL(X,\om)$, and Veech's counting results. We note that a generic (in the sense of Masur-Smillie-Veech-almost every) surface has \emph{no} pairs of parallel holonomy vectors, as the existence of such a pair is a closed  and positive codimension condition.

\paragraph*{\bf Parallel saddle connections} We start our study of counting pairs of saddle connections by proving Proposition~\ref{theorem:parallel}. Let $\om$ be a lattice surface. We recall first further details of Veech's result on the decomposition of $\La_{\om}$ into finitely many orbits of the Fuchsian group $\Gamma = SL(\om)$ acting $\R$-linearly on $\C$.

\subsection{Cusps and orbits}\label{sec:cusps} We denote the finite collection of \emph{cusps} of $\Gamma$ by $[\Gamma_1], [\Gamma_2], \ldots [\Gamma_n]$, where each $[\Gamma_i]$ is a distinct conjugacy class of a parabolic subgroup of $\Gamma$. To each $[\Gamma_i]$ we can choose a direction in $\R \cup \{\infty\}$ stabilized by a representative of $[\Gamma_i]$, and in this direction, there will be a finite set of parallel saddle connections $\gamma_{i, 1}, \ldots, \gamma_{i, m_i}$ with $$\ell(\gamma_{i, 1}) \geq \ell(\gamma_{i, 2}) \ldots \geq \ell( \gamma_{i, m_i}).$$ Let $$z_{i, j} = \int_{\gamma_{i,j}} \omega, r_{i,k} = \frac{z_{i,1}}{z_{i,k}}.$$ Note that for a fixed $i$, the $z_{i,j}$ are parallel, so the $r_{i,k} = \frac{z_{i,1}}{z_{i,k}} = \frac{\ell(\gamma_{i,1})}{\ell(\gamma_{i,k})}$ are real numbers greater than $1$ for $k \geq 1$. By~\cite[Theorem 16.1]{Veech98}, there is a $c_{i,k}$ such that $$\# (\Gamma \cdot z_{i,k}  \cap B(0, R)) \sim c_{i,k} R^2.$$ Note that since $$\# (\Gamma \cdot z_{i,k}  \cap B(0, R)) = \# (\Gamma \cdot z_{i,1}  \cap B(0, Rr_{i,k})) ,$$ we have $$c_{i, k} = c_{i, 1} r_{i,k}^2.$$ Therefore, for $j <k$, $$\# (\Gamma \cdot (z_{i,k}, z_{i,j})  \cap B(0, R)^2) = \# (\Gamma \cdot z_{i,j}  \cap B(0, R)) \sim c_{i, 1} r_{i,j}^2 R^2.$$ \subsection{Completing the proof} To complete the proof of Proposition~\ref{theorem:parallel}, we put these together to obtain $$N_0(\omega, R) = \sum_{i=1}^n \sum_{j=1}^{m_j-1} \sum_{k=1}^{j-1} \# (\Gamma \cdot (z_{i,k}, z_{i,j})  \cap B(0, R)^2) \sim \sum_{i=1}^n \sum_{j=1}^{m_j-1} \sum_{k=1}^{j-1} c_{i,1} r_{i,j}^2 R^2.$$ This proves Proposition~\ref{theorem:parallel}, with $$c = \sum_{i=1}^n \sum_{j=1}^{m_j-1} \sum_{k=1}^{j-1} c_{i,1} r_{i,j}^2 = \sum_{i=1}^n c_{i,1}\sum_{j=1}^{m_j-1} (j-1)  r_{i,j}^2.$$ \qed

\section{Nevo's ergodic theorem and Siegel--Veech measures}\label{sec:spectral} 
\paragraph*{\bf Counting asymptotics and averaging operators} We state the results of \cite{nevo2017equidistribution} and \cite{AthreyaCheungMasur19} needed to precisely move between circle averages and counting asymptotics. We then combine these results to show convergence of averaging operators of Siegel--Veech transforms.
\subsection{Averaging operators}\label{sec:average} Suppose $SL(2, \R)$ acts on a space $X$ (here our spaces will be $\C$, $\C^2$, and connected components of strata $\hh$, all with the natural $\R$-linear actions). Given a function $h$ on $X$, and $p \in X$, we define $$\left(A_t h\right)(p) = \frac{1}{2\pi} \int_0^{2\pi} h(g_t r_{\theta} p) \, d\theta.$$ Note that for $f \in B_c(\C)$ or $h \in B_c(\C^2)$, we can interchange sum and integral to obtain $$(A_t \widehat{f})(\om) = \widehat{(A_t f)}(\om) \mbox{ and } (A_t \widehat{h})(\om) = \widehat{(A_t h)}(\om).$$

\paragraph*{\bf Nevo's ergodic theorem} A  key tool is Nevo's ergodic theorem for the operators $A_t$ acting on $\mathcal H$.   To state the theorem we first need the definition of $K$-finite function.

\paragraph*{\bf K-finite functions} Given a space $X$ on which $SL(2,\R)$ acts, we say a function $f$ defined on $X$ is \textit{$K$-finite} if  the span of the functions $\{f\circ r_{\theta} : \theta \in [0, 2\pi)\}$ has finite dimension. Equivalently, there is an integer $m$ such that  for each $x\in X$ the function of $\theta$ defined by $p_{f, x}(\theta) = p(\theta)=f(r_\theta(x))$ is a trigonometric polynomial of  degree at most $m$.

  \begin{Theorem}\label{thm:Nevo}\cite[Theorem 1.1]{nevo2017equidistribution} Suppose $\mu$ is an ergodic $SL(2, \R)$-invariant probability measure on $\hh$. Assume $f \in L^{1+\kappa}(\hh, \mu)$ for some $\kappa >0$, and that $f$ is $K$-\emph{finite}. Let $\eta \in C_c(\R)$ be a continuous non-negative bump function with compact support and of unit integral. Then for $\mu$-almost every $\omega \in \hh$, $$\lim_{t \rightarrow \infty} \int_{-\infty}^{\infty} \eta(t-s) (A_s f)(\omega)ds = \int_{\hh} f d\mu.$$

\end{Theorem}


\subsection{Siegel--Veech measures}\label{sec:svmeasures} We will apply Nevo's theorem to the Siegel--Veech transforms of functions defined on $\mathbb{C}^2$.  By ~\cite[Theorem 3.4]{AthreyaCheungMasur19}, for any $f \in B_c(\C)$, for some $\kappa>0$, $\widehat{f} \in L^{2+2\kappa}(\hh, \mu)$. Then, for any $h \in B_c(\C^2)$, $\widehat{h} \in L^{1+\kappa}(\hh, \mu)$, since we can dominate  $$\widehat{h}(\omega) = \sum_{v_1, v_2 \in \Lambda_{\omega}} h(v_1, v_2)$$ by $(\widehat{f})^2$ where $f = \|h\|_{\infty} \chi_{H}$, where $H$ denotes the union of the projections of the support of $h$ via the coordinate projection maps. By the invariance of $\mu$, and the integrability condition, $$h \longmapsto \int_{\mathcal H} \widehat{h}(\omega)d\tau(\omega)$$ is an $SL(2,\R)$-invariant linear functional on $C_c(\C^2)$. Therefore, there is an $SL(2, \R)$-invariant measure $m = m(\mu)$ (a \emph{Siegel--Veech measure}) on $\C^2 $  so that $$\int_{\mathcal H} \widehat{h}(\omega)d\mu(\omega)= \int_{\C^2} h \, dm$$
 By the monotone convergence theorem, we can extend the class of $h$ for all $h\in B_c^{SC}(\C^2)$, which consists of functions $h \in B_c(\C^2)$ which are either upper or lower semi-continuous. In particular $B_c^{SC}$ will include characteristic functions of the compact closed sets defined in Section~\ref{sec:approx}. 
 
\paragraph*{\bf Invariant measures on $\C^2$} To describe the possible $SL(2,\R)$-invariant measures on $\C^2$, we need to understand $SL(2, \R)$-orbits on $\C^2$. For $t \in\R$, let $$D_t = \{ (z, w) \in 
\C^2:|\det(z | w) = t\},$$
and notice we can identify $D_1$ with  $SL(2,\R)$.
For $t \neq 0$, $D_t$ is an $SL(2, \R)$-orbit. The determinant zero locus $D_0$ decomposes further. For $s \in \mathbb P^1(\R)$, let  $$L_s = \{ (z, sz): z \in \C \backslash\{0\}\},$$ with $$L_{\infty} = \{ (0, w): w \in \C \backslash\{0\}\}.$$ The sets $D_t$ and $L_s$ are the non-trivial $SL(2,\R)$ orbits on $\C^2$, and each carries a unique (up to scaling) $SL(2, \R)$-invariant measure. These are the (non-atomic) ergodic invariant measures for $SL(2,\R)$ action on $\C^2$. On $D_t$, the measure is Haar measure (which we denote $\lambda$) on $SL(2, \R)$, and on $L_s$ it is Lebesgue on $\C$. Thus, associated to any $SL(2,\R)$ invariant measure $m$ on $\C^2$ we have measures $\nu = \nu(m)$ and $\rho = \rho(m)$ so that we have~\cite[Theorem 1.2]{AthreyaCheungMasur19}
\begin{equation} \label{eq:SV-measure}\int_{\C^2} h \,dm = \int_{\R\minuszero} \left(\int_{SL(2, \R)} h(tz, w) d\lambda(z,w) \right) d\nu(t) + \int_{\mathbb P^1(\R)} \left(\int_{\C} h(z, sz) dz\right) d\rho(s).
\end{equation} 

\subsection{Convergence of averaging operators for continuous functions}

\subsubsection{Notation}\label{sec:SV:notation} The goal of this section is Proposition~\ref{prop:convergescont}  which establishes convergence of circle averages for continuous functions.  Later we will establish the same result for the function $h_A$.  In the course of the proof (and in later sections of the paper) the functions we are taking transforms of will be sometimes quite complicated to write down. We introduce the following notation: given bounded compactly supported functions $f \in B_c(\C)$ or $h \in B_c(\C^2)$ we write $$\widehat{f}(\om) = f^{\SV}(\om) \mbox{ and } \widehat{h}(\om) = h^{\SV}(\om).$$

\begin{Prop}
\label{prop:convergescont}
Suppose $\varphi \in C_c(\C^2)$. Suppose the support does not contain points of the form  $(z,0)$ or $(0,w)$.
Then for $\mu$-almost every $\omega \in \hh$, the circle averages of $\widehat{\varphi}$ converge $$\lim_{\tau\to\infty} A_\tau \widehat \varphi(\omega) = \int_{\mathcal{H}} \widehat{\varphi} \,d\mu = \int_{\C^2} \varphi \,dm.$$ 
\end{Prop}

\begin{proof}

By~\cite[Theorem 3.4]{AthreyaCheungMasur19}, $\widehat \varphi \in L^{1+\kappa}(\hh,\mu)$ for some $\kappa>0$. We want to construct $K$-finite functions which sufficiently approximate $\widehat{\varphi}$, which we do by constructing a family of $K$-finite functions which are dense in the continuous functions.

\paragraph*{\bf Products of annuli} For  $r_1>r_0>0$ and $N\geq 2$, define $H= H(r_0, r_1, N) = \prod_{i=1}^N \overline{A(r_0,r_1)}$ to be a product of $N$ annuli.   The group $K$ acts on $\C^N$ diagonally, preserving $H$, via $$r_{\theta}(z_1,\ldots, z_N) = (r_{\theta}z_1,\ldots  r_{\theta}z_N).$$ Let $C(H)$ denote the set of continuous functions of $H$, and consider the family $\mathcal F \subset C(H)$ given by 

$$\mathcal{F}=\{ f_{m_1,n_1,m_2,n_2,\ldots, m_N.n_N}: m_i, n_i \in \Z \}$$ where for $z_j= r_j e^{i\theta_j},$ 
$$ f_{m_1,n_1,\ldots m_N,n_N}(z_1,\ldots, z_N) =\prod_{j=1}^N r_j^{m_j}e^{in_j\theta_j}.$$  

\noindent By definition, the functions $f\in\mathcal{F}$ are $K$-finite.
\paragraph*{\bf A distinguished subalgebra} We consider a subalgebra $\mathcal{A}$ of $C(H)$ given by the $\C$-linear span of $\mathcal{F} \cup \{1\}$, where $1$ is the constant unit function on $H$. Then by definition, $\mathcal{A}$ is closed under addition and multiplication by complex scalars. Moreover $\mathcal{F}$ is closed under multiplication and complex conjugation, so $\mathcal{A}$ is an algebra. 
 \begin{lemma}
\label{lem:separate}
The collection $\mathcal{F}$ separates points on $H$, meaning if $p\neq q$ are points of $H$ then there exists $f\in\mathcal{F}$ such that $f(p)\neq f(q)$. 
\end{lemma}
\begin{proof}
 Let $(z_1,\ldots, z_N) =
 \left(r_1 e^{i\theta_1}, \ldots, r_Ne^{i\theta_N}\right)$ 
and $\left(\tilde{z}_1,\ldots, \tilde{z}_N\right)= \left(\tilde{r}_1 e^{i\tilde{\theta}_1}, \ldots, \tilde{r}_Ne^{i\tilde{\theta}_N}\right)$ 
be two distinct points in $H$.  Suppose first case that $\prod_{j=1}^N r_j\neq \prod_{j=1}^N\tilde{r}_j $. Then we have $$f_{1,0,1,0, \dots, 1,0} (z_1,\ldots, z_N) \neq f_{1,0,1,0,\ldots, 1,0}(\tilde{z}_1,\ldots, ,\tilde{z}_N).$$
Now suppose $\prod_{j=1}^Nr_j=\prod_{j=1}^N  \tilde{r}_j$. Let us assume first that for some $j$  it is the case that $r_j\neq \tilde{r}_j$.  Without loss of generality assume $j=N$. 
Then since $r_N,\tilde{r}_N\neq 0$, $$r_N\prod_{j=1}^N r_j\neq \tilde{r}_N\prod_{j=1}^N \tilde{r}_j$$ so $f_{1,0,\ldots, 1,0, 2,0}$ separates the two points. 
Now assume $r_j=\tilde{r}_j$ for all $j$.  
Then we have to use the angles to separate the points.  The first subcase is $$\sum_{j=1}^N \theta_j\neq \sum_{j=1}^N \tilde{\theta}_j.$$ In that case $f_{0,1,\ldots, 0,1}$ separates the two points. 
 Thus assume  $$\sum_{j=1}^N \theta_j=\sum_{j=1}^N \tilde{\theta}_j.$$  Since the points are distinct and $r_j=\tilde{r}_j$ for all  $j$ we must have  $\theta_j\neq \tilde{\theta}_j$ for some $j$. Without loss of generality assume $j=N$. Then $$\sum_{j=1}^N \left(\theta_j+\theta_N \right) \neq \sum_{j=1}^N \left(\tilde{\theta}_j +\tilde{\theta}_N\right)$$ and so $f_{0,1,\ldots, 0,1,0,2}$ separates. 
\end{proof}

\paragraph*{\bf Stone-Weierstrass} Thus by the Stone-Weierstrass Theorem~\cite[Theorem 1.26]{Rudin}, this set of $K$-finite functions is dense in the uniform topology in $C(H)$.  Now we fix $N=2$ and a translation surface $\omega$.  We choose $r_0,r_1$ in the definition of $H= H(r_0, r_1)$ so that $H$ contains the support of $\varphi$, so  $\varphi \leq \norm{\varphi}_{\infty} \one_{H}.$
We will also consider also the slightly larger set $H_1 = H(r_0/2, r_1+1) =  \overline{A(\frac{r_0}{2},r_1+1)}^2$. Notice that $H, H_1$ are rotation invariant subsets of $\C^2$ under the diagonal action of $r_\theta$. 

\paragraph*{\bf $K$-finiteness} We next see that Siegel--Veech transforms of elements of $\mathcal{F}$ are $K$-finite.
This follows from the fact that for each $\om\in\mathcal{H}$,  $\hat f(\om)$ is a finite sum of values of $f$, so the corresponding function $p(\theta) = p_{\widehat{f}, \om}(\theta)$ in the definition of $K$-finiteness for $\widehat{f}$ is a finite sum of trigonometric polynomials of degree at most $m$ by the $K$-finiteness of $f$; hence a polynomial of degree $m$ as well.

\paragraph*{\bf Mollifiers} Fix  $\eta\in C_c^\infty(\mathbb{R})$ a positive mollifier, with $\eta(t) \geq 0$, $$\int_{-\infty}^\infty \eta(t)  \,dt= 1,$$ and support of $\eta$ in $[-1,1]$.  For $f\in C_0^\infty(\mathbb{C}^2)$, and $(z, w) \in \C^2$, denote the convolution by $$(\eta*f)(z,w):= \int_{-\infty}^\infty \eta(t)f(g_{-t}(z,w))dt.$$ We will use the notation $\eta_\gamma(t) = \gamma^{-1} \eta(t/\gamma)$ which has the property that the support of $\eta_\gamma$ is in $[-\gamma, \gamma]$ and
$$\lim_{\gamma \to 0} \eta_\gamma(t) = \delta(t)$$ where $\delta$ is the Dirac delta distribution.  

\paragraph*{\bf Uniform convergence} We claim  $\eta_\gamma * \varphi $ converges uniformly to $\varphi$ on $\C^2$ as $\gamma \to 0$. To see this, for any $\epsilon > 0$ choose $\gamma_0$ so that whenever $\gamma \leq \gamma_0$, since the support of $\eta_\gamma$ is contained in $[-\gamma, \gamma]$, by uniform continuity of $\varphi$ on $H$, whenever $|t|< \gamma$, 
$$|\varphi(g_{-t}(z,w)) - \varphi (z,w)| < \epsilon.$$
Thus  for any $(z,w) \in \C$,
$$|(\eta_\gamma * \varphi)(z,w) - \varphi(z,w)| \leq \int_{-\gamma}^\gamma \eta_{\gamma}(t) |\varphi(g_{-t}(z,w) - \varphi(z,w)| \,dt < \epsilon.$$
We will also use the fact that there is some $\gamma_0$ so that for $\gamma<\gamma_0$, for any $(z,w) \in H$, $g_{-\gamma} (z,w)\in H_1$, so 
\begin{equation}
\label{eq:convolve}
\one_H \leq \eta_{\gamma_0} * \one_{H_1}.
\end{equation}

\noindent By Theorem~\ref{thm:Nevo}, for each $f_n$ and almost every $\omega$, 
\begin{align*} 
\lim_{\tau\to\infty} A_\tau(\widehat{\eta * f_n})(\omega) &= \lim_{\tau\to\infty} \int_{-\infty}^\infty\eta(t) \left(A_{\tau-t}\widehat{f_n}\right)(\om) dt \\ 
&=
\lim_{\tau\to\infty}\int_{-\infty}^{\infty}\eta(\tau-s)\left(A_s\widehat{f_n}\right)(\om) ds\\
& = \int_{\mathcal{H}} \widehat{f}_n \,d\mu = \int_{\C^2} f_n \,dm,
\end{align*}
where $m$ is the Siegel--Veech measure as in Equation~\eqref{eq:SV-measure}.

\paragraph*{\bf A triangle inequality argument} Now that we've established the properties needed for $K$-finite functions, we want to use a triangle inequality argument to get bounds on the difference $|A_{\tau}\widehat{\varphi}(\omega) - \int_{\C^2} \varphi \,dm|$. Fix $\epsilon > 0$. First we use uniform convergence to obtain control over the Siegel--Veech transforms. Choose $H$ slightly larger if necessary so all $\omega$ have a saddle connection in $B(0,l_{\varphi})$, so $$1\leq \widehat{\one_H}(\omega) < \infty.$$
\noindent Since $\one_H \in B_c^{SC}(\C^2)$, we have $\widehat{\one_H} \in L^{1+\kappa}(\mu)$ so
$$\int_\mathcal{H} \widehat{ \one_H} \, d\mu = m(H) < \infty.$$ By uniform convergence, choose $N$ so that for all $n\geq N$ and a corresponding sequence $\gamma_n$, for all $(z,w) \in H$,
\begin{equation}\label{eq:0unif}
|f_n(z,w) - \varphi (z,w)| < \epsilon \text{ and } |(\eta_{\gamma_n} * \varphi)(z,w) - \varphi(z,w)| < \epsilon.
\end{equation}

\paragraph*{\bf Differences of $SV$-transforms} Now that we have sufficient convergence, we now consider differences of Siegel--Veech transforms. Since the support of $f_n$ and $\varphi$ are contained in $H$, summing over $\Lambda_{\omega'}$ for any $\omega' \in \mathcal{H}$ gives the pointwise bounds
\begin{equation} \label{eq:0.25unif}|\widehat{f_n}(\omega') - \widehat{\varphi}(\omega')| < \epsilon \widehat{\one_H}(\omega')\text{ and } |(\eta_{\gamma_n} * \varphi)^{\SV}(\omega') - \widehat{\varphi}(\omega')| <\epsilon \widehat{\one_H}(\omega').
\end{equation}

\paragraph*{\bf Differences of circle averages}  Now we apply  Theorem~\ref{thm:Nevo} to get differences of circle averages. Fix $n$ large. Since $\one_{H_1}$ is $K$-finite, we apply Theorem~\ref{thm:Nevo}. So for almost every $\omega \in \mathcal{H}$ choose $T = T(\omega)$ large enough so that for all $\tau \geq T$,  
\begin{equation}\left|{A_\tau(\eta_{\gamma_n}*\one_{H_1})^{\SV}(\omega) - m(H_1)}\right|< \epsilon.
\end{equation}

 Choose $\tau \geq T$ for $T = T(\eta_{\gamma_0}, \omega)$ so that again by Theorem~\ref{thm:Nevo} \begin{equation}
\label{eq:033unif} A_\tau \widehat{\one_H(\omega)} \leq A_\tau (\eta_{\gamma_n} * \one_{H_1})^{\SV}(\om) \leq m(H_1) + \epsilon.
\end{equation}

\paragraph*{\bf Putting it together} We now are able to obtain bounds for the four differences used in the triangle inequality argument. For almost every $\omega$, and $\tau \geq T(\omega)$, we  apply Equation~\eqref{eq:0.25unif} for each $\omega'=g_\tau r_\theta \omega$, and then by Equation~\eqref{eq:033unif}
\begin{align} \label{eq:0.5unif}|A_\tau(\widehat{\eta_{\gamma_n}*f_n})(\omega) - A_\tau(\eta_{\gamma_n}*\varphi)^{\SV}(\omega)| &\leq \epsilon A_\tau(\eta_{\gamma_n} * \one_{H_1})^{\SV}(\omega) \\ \nonumber &< \epsilon\left[\epsilon + m(H_1)\right].
\end{align}

\noindent Equation \eqref{eq:0.25unif} also says that for each $\omega' = g_t r_\theta \omega$, we obtain
\begin{align} \label{eq:1unif}
 |A_\tau(\eta_{\gamma_n} * \varphi)^{\SV}(\omega) - A_\tau\widehat{\varphi }(\omega)| &\leq  A_\tau |(\eta_{\gamma_n} * \varphi)^{\SV}(\omega) - \widehat{\varphi}(\omega)| \\ \nonumber & \leq \epsilon A_\tau(\widehat{\one_H})(\omega) \\ \nonumber &\leq \epsilon[\epsilon+ m(H_1)].
\end{align}

\noindent Applying Theorem~\ref{thm:Nevo}, for a.e.~$\omega$ choose $T = T(\omega)$ large enough so that for all $\tau \geq T$,
\begin{equation}\label{eq:2unif}
\abs{A_\tau(\eta_{\gamma_n}*f_n)^{\SV}(\omega) - \int_{\C^2} f_n \, dm} < \epsilon.
\end{equation}
Now we use Equation~\eqref{eq:0unif} and the fact that the functions have support in $H_1$ to see that 
\begin{equation}\label{eq:3unif}
	\abs{\int_{\C^2} f_n\,dm - \int_{\C^2} \varphi\,dm} \leq \epsilon \, m(H_1).
\end{equation}

\noindent By the triangle inequality combined with Equations~\eqref{eq:0.5unif},
\eqref{eq:1unif},
 \eqref{eq:2unif}, \eqref{eq:3unif} we conclude that for almost every $\omega$ and each $n$, there is $T= T(\varphi,\omega)$ so that for all $\tau \geq T$,
\begin{equation} \label{eq:gnconv0}
\abs{A_{\tau} \widehat{\varphi} (\omega)- \int_{\C^2} \varphi \,d m} \leq 2\epsilon[\epsilon+ m(H_1)] + \epsilon + \epsilon m(H_1).
\end{equation}

\noindent Since $m(H)$ and $m(H_1)$ are fixed constants and $\epsilon$ is arbitrary  we conclude for almost every $\omega$,
\begin{equation}\label{eq:gnconv}
	\lim_{\tau\to\infty}  A_\tau\widehat{\varphi}(\omega) = \int_{\C^2} \varphi \,dm,
\end{equation}
\noindent concluding the proof of Proposition~\ref{prop:convergescont}.

\end{proof}
\section{Approximation function and properties}\label{sec:approx}  \paragraph*{\bf Constructing $h_A$} We now construct the function $h_A$ satisfying (\ref{eq:circle}).

\subsection{Fibered sets}\label{sec:fibered}

Fix $A>0$. Given $z \in \C$, define the approximating parallelogram $$R_A(z) = \{w \in \C: |w \wedge z| \le A, |\im w| \le |\im z|\}$$ which we will use to approximate the desired set $$D_A(z) =  \{w \in \C: |w \wedge z| \le A, |w| \le |z|\}.$$

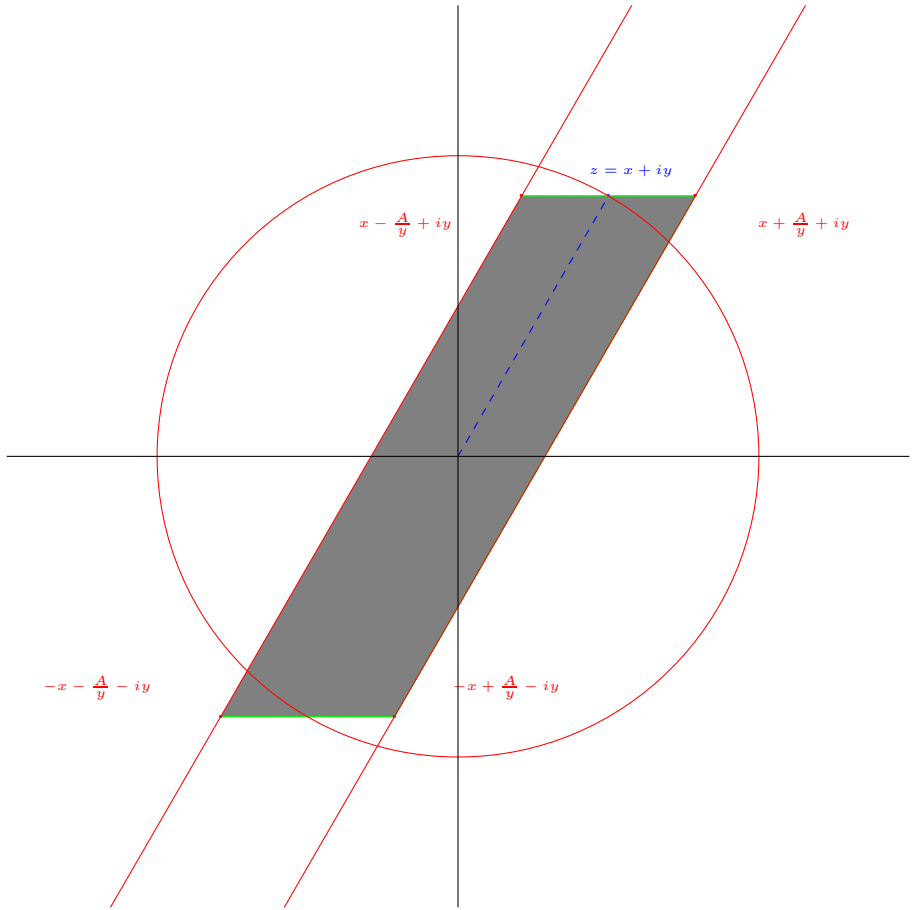
\begin{figure}[h!]\caption{The sets $R_A(z)$ and $D_A(z)$. $R_A(z)$ is the shaded parallelogram, and $D_A(z)$ is the region inside the red circle bounded by the two parallel red lines.  This is the picture for $x>A/y$, and there is a corresponding figure for $x<A/y$.\medskip}\label{fig:rada}

\begin{tikzpicture}[scale = 2.0]

\filldraw[gray](-1.577, -1.732)node[red]{$\cdot$}--(-.423, -1.732)node[red]{$\cdot$}--(1.577,1.732)node[red]{$\cdot$}--(.423, 1.732)node[red]{$\cdot$};

\draw[dashed, blue](0,0)--(1, 1.732)node[blue]{$\cdot$};

\node at (1.15,1.8)[blue, right, above]{\Tiny $z= x+iy$};

\draw[green](-1.577, -1.732)node[red]{$\cdot$}--(-.423, -1.732)node[red]{$\cdot$}--(1.577,1.732)node[red]{$\cdot$}--(.423, 1.732)node[red]{$\cdot$};

\draw[red](-1.155, -3)--(2.31, 3);

\draw[red](-2.31, -3)--(1.155, 3);

\draw[red](0,0) circle (2);

\node at (2.3,1.4)[red, right, above]{\Tiny $x+ \frac{A}{y} +iy$};

\node at (-.35,1.4)[red, left, above]{\Tiny $x- \frac{A}{y}+iy$};

\node at (-2.4,-1.4)[red, left, below]{\Tiny $-x- \frac{A}{y} -iy$};

 \node at (.32, -1.4)[red, right, below]{\Tiny $-x+ \frac{A}{y}-iy$};

\draw (-3,0)--(3,0);

\draw (0, -3)--(0, 3);

\end{tikzpicture}
\end{figure}

\paragraph*{\bf The fibered parallelogram} Given $S \subset \C$, we define the \textit{fibered} parallelogram and desired set by \begin{align*}R_A(S) &= \{(z, w) \in \C^2: z \in S, w \in R_A(z)\}\text{  and  } D_A(S) = \{(z, w) \in \C^2: z \in S, w \in D_A(z)\}. \end{align*}

\paragraph*{\bf Equivariance} Note that for $z=x+iy$, $R_A(z)$ is a parallelogram with vertices $x \pm \frac{A}{y} \pm  iy$ (see Figure~\ref{fig:rada}). From this observation, we have $g_t$-equivariance of $R_A$: for $t \in \R$, $$g_t(R_A(z)) = R_A(g_t z),$$ so for $S \subset \C$, $$g_t(R_A(S)) = R_A(g_t S).$$ We will be particularly interested in two families of fibered sets. First, define the trapezoid $\T$ by

$$\T = \left\{z = x+iy \in \C:  \frac{1}{2} \leq y \leq 1, \, |x|\leq y \right\}$$
We set $h_A = \one_{R_A(\T)}$ to be the indicator function of the fibered set $R_A(\T)$. Next, for $R_2 > R_1> 0$ set \begin{align*} B(R_1) &= B(0, R_1) = \{z \in \C: |z| < R_1\} \\ A(R_1, R_2) & = B(R_2) \backslash B(R_1) = \{z \in \C: R_1 < |z| < R_2\}\end{align*}
	
	 \noindent We define \begin{align*} D_A(R) &= D_A(B(R)) \\ D_A(R_1, R_2) &= D_A(A(R_1, R_2)).\end{align*} We have $N_A(\omega, R) = (\one_{D_A(R)})^{\SV}(\omega)$ and define $N_A^*(\omega, R) =  (\one_{D_A(R/2, R)})^{\SV}(\omega)$. Our main goal is to prove \begin{Theorem}\label{thm:error}
	\begin{equation} \label{eq:approx}
	\left|N_A^*(\omega, e^t) - \pi e^{2t}\left(A_t \widehat{h_A}\right)(\omega) \right| = o(e^{2t}).
	\end{equation}
	 \end{Theorem}
	\noindent In order to obtain Theorem~\ref{thm:error}, we will need to use quadratic upper bounds to control error terms.
	\begin{Prop}
\label{prop:quadratic}
Given $A>0$, there exists $C$ such that for a.e. $(X,\omega)$ there exists $T>0$ such that for all $t>T$,  
$$N_A^*(\omega, e^t) \leq C e^{2t}.$$
\end{Prop}

\noindent We will also need the following proposition.  
 \begin{Prop}\label{prop:hAlimit}
	For almost every $\omega$,
	$$\lim_{t\to\infty} (A_t \widehat{h_A})(\omega)  = \int_{\mathcal{H}} \widehat{h_A} \,d\mu.$$
\end{Prop}	
\noindent To prove this proposition, we will apply  Proposition~\ref{prop:convergescont} where we proved this for continuous $\phi$. We will use the following lemma to construct a sufficiently good approximating function $g_\epsilon$ which is independent of the time $\tau$ of flow under the circle averages. 
\begin{lemma}\label{lem:smallboundary}
	For all $\epsilon$ there exists a function $g_\epsilon \in C_c(\C^2)$ such that for $\mu$-a.e $\omega$ there is $T \geq 0$ so that for all $\tau \geq T$, 
	\begin{equation} \label{eq:geps}
		\abs{A_\tau(\widehat{g_\epsilon} - \widehat{h_A})(\omega)} < \epsilon\quad \text{and}\quad \abs{\int_{\mathcal{H}} \widehat{g_\epsilon} - \widehat{h_A} \,d\mu} < \epsilon.
\end{equation}
\end{lemma}
We postpone the proof of this lemma until Section~\ref{sec:upperbound}. The reason for postponing is first that the proof requires careful analysis of the thin part of the stratum, which requires splitting into three subsets $F,G,H$ depending on the length of the shortest saddle connection and a second non-homologous saddle connection (See Lemma~\ref{lem:integralthin}). In the thick part of the stratum, we also need to develop measure bounds for $\epsilon$-small sets (See Lemma~\ref{lem:bad}). Finally we need to take care in the first inequality of \eqref{eq:geps} to make sure $g_\epsilon$ is independent of $\tau$.

\begin{proof} [Proof of Proposition~\ref{prop:hAlimit}]
	Let $\epsilon > 0$. By Lemma~\ref{lem:smallboundary}  we may choose a function $g_\epsilon \in \C_c(\C^2)$ so that  Equation~\ref{eq:geps} holds. 	  Now choose $T=T(\omega,\epsilon)$  so that by Proposition~\ref{prop:convergescont},
 for $\tau\geq T$, $$\left|A_\tau \widehat{g_\epsilon}(\omega) - \int_{\mathcal{H}} \widehat{g_\epsilon} \,d\mu\right| < \epsilon.$$	
	Thus by the triangle inequality
	$$\abs{A_\tau \widehat{h_A}(\omega) - \int_{\mathcal{H}} \widehat{h_A} \,d\mu}  \leq  \abs{A_\tau \widehat{h_A}(\omega) - A_\tau \widehat{g_\epsilon}(\omega)} + \abs{A_\tau \widehat{g_\epsilon}(\omega) - \int_{\mathcal{H}} \widehat{g_\epsilon} \,d\mu } + \abs{ \int_{\mathcal{H}} \widehat{g_\epsilon} - \widehat{h_A} \,d\mu } < 3\epsilon.$$
\end{proof}

\paragraph*{Proving our main result} We conclude this section by proving Theorem~\ref{theorem:virtual:ae} assuming Theorem~\ref{thm:error}, Proposition~\ref{prop:quadratic} and  Lemma~\ref{lem:smallboundary}. We note  Proposition~\ref{prop:hAlimit} has been  proved assuming that Lemma~\ref{lem:smallboundary} holds. The proofs  of Theorem~\ref{thm:error}, Proposition~\ref{prop:quadratic} and Lemma~\ref{lem:smallboundary} are contained in Section~\ref{sec:upperbound}

\begin{proof}[Proof of Theorem~\ref{theorem:virtual:ae}]  
Combining Theorem~\ref{thm:error} and Proposition~\ref{prop:hAlimit}
we have, that for $\mu$-almost every $\omega \in \hh$ \begin{equation} \label{eq:N_A^*}\lim_{t \rightarrow \infty}\frac{N_A^*(\omega, e^t)}{\pi e^{2t}} = c_0
\end{equation} where
$$c_0 =\int_{\R\minuszero} \left(\int_{SL(2, \R)} h_A(tz, w) d\lambda(z,w) \right) d\nu(t)+\int_{\mathbb P^1(\R)} \left(\int_{\C} h_A(z, sz) dx\right) d\rho(s).$$ 
Notice that $c_0$ only depends on $A$ and $(\mathcal{H}, \mu)$.

\paragraph*{\bf Geometric series} To extend to $N_A(\omega, e^t)$ we use a geometric series argument along with the dominated convergence theorem giving upper bounds via Proposition~\ref{prop:quadratic}. Specifically for each fixed $j$ setting $s = \frac{e^t}{2^j}$ and using Equation~\ref{eq:N_A^*}, we have pointwise convergence
$$\lim_{t\to\infty}\frac{N_A^*(\omega, \frac{e^t}{2^j})}{e^{2t}} = \lim_{s\to\infty} \frac{N_A^*(\omega, e^s)}{2^{2j}e^{2s}} = \frac{\pi c_0}{2^{2j}}. $$

\noindent We also have a dominating integrable function $$\frac{N_A^*(\omega, \frac{e^t}{2^j})}{e^{2t}} \leq \frac{c T^2}{2^{2j}},$$
where we are without loss of generality assuming $T>1$ is the constant from Proposition~\ref{prop:quadratic}. Namely for each $j$ whenever $e^t > T2^j$, Proposition~\ref{prop:quadratic} gives an upper bound of $C 2^{-2j} < cT^2 2^{-2j}$. If $e^t \leq T 2^j$, using quadratic upper bounds from Equation~\ref{eq:quadbounds}, 
$$\frac{N_A^*(\omega, \frac{e^t}{2^j})}{e^{2t}} \leq \frac{N(\omega,\frac{e^t}{2^j})^2}{e^{2t}} \leq c_2 \frac{e^{2t}}{2^{4j}} \leq c_2\frac{T^2}{2^{2j}}.$$

\noindent Therefore by the dominated convergence theorem and the fact that for each fixed $t$, the tail of the telescoping series gives $$\lim_{j\to\infty} N_A\left(\omega, \frac{e^t}{2^{j+1}}\right) = 0,$$ so
$$\lim_{t\to\infty} \frac{N_A(\omega, e^t)}{e^{2t}} = \lim_{t\to\infty} \sum_{j=0}^\infty \frac{N_A^*(\omega, \frac{e^t}{2^j})}{e^{2t}} = \sum_{j=0}^\infty \frac{\pi c_0}{2^{2j}} = \frac{4}{3} \pi c_0.$$
\end{proof}

\section{Counting and Errors}\label{sec:estimates}

 \paragraph*{\bf Error estimates} In this section we derive estimates that are necessary steps to prove Theorem~\ref{thm:error}. The techniques are elementary and independent of Section~\ref{sec:upperbound}. Note that (see Figure~\ref{fig:arcs}) the trapezoid $g_{-t}\mathcal{T}$ has vertices $$\pm \frac{e^{-t}}{2} + i \frac{e^{t}}{2},\quad \pm e^{-t}+ i e^t.$$ Given $z, w \in \C$, define \begin{align*} \Theta_t(z) &= \{\theta \in [0, 2\pi): g_t r_{\theta} z \in \T\} \\ \Theta_t(z, w) &=\{\theta \in [0, 2\pi): g_t r_{\theta}(z, w) \in R_A(\T)\}.\end{align*} For $f = \one_{\T} $ and $h = \one_{R_A(\T)}$, we have \begin{align*} |\Theta_t(z)| &= 2\pi (A_t f)(z) \\ |\Theta_t(z, w)| &= 2\pi(A_t h_A)(z, w)\end{align*}

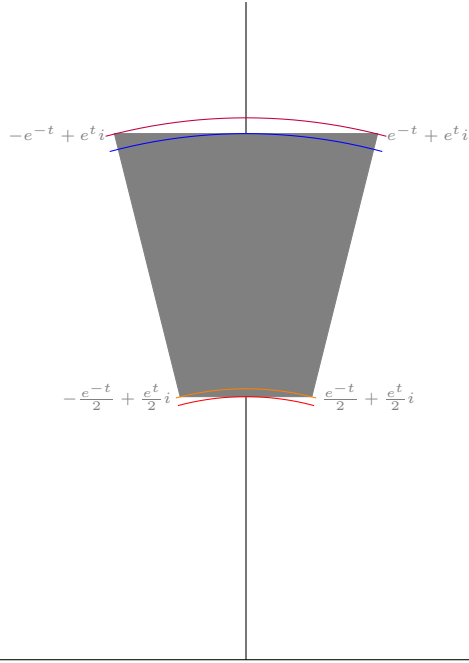
\begin{figure}[hb]\caption{Here we highlight the differences that occur when looking at the trapezoid $g_{-t}\mathcal{T}$ as compared to circle arcs that touch the trapezoid. In particular the blue and orange arcs illustrate the fact that for $z \in A\left( \sqrt{\frac{\cosh(2t)}{2}}, e^t\right)$, $|\Theta_t(z)| = 2 \arctan(e^{-2t})$. }\label{fig:arcs}

\centering
\begin{tikzpicture}[scale = 3.5]

\draw(-2,0)--(2,0);

\draw(0, 0)--(0, 2.5);

\filldraw[gray](1/4, 1)node[right]{\tiny $\frac{e^{-t}}{2} + \frac{e^t}{2} i$}--(1/2, 2)node[right]{\tiny $e^{-t}+ e^t i$}--(-1/2, 2)node[left]{\tiny $-e^{-t}+ e^t i$}--(-1/4, 1)node[left]{\tiny $-\frac{e^{-t}}{2} + \frac{e^t}{2} i$}--cycle;
   \draw [red,domain=75:105] plot ({cos(\x)}, {sin(\x)});
   \draw [blue,domain=75:105] plot ({2*cos(\x)}, {2*sin(\x)});
   \draw [orange,domain=75:105] plot ({(1.03)*cos(\x)}, {(1.03)*sin(\x)});
   \draw [purple,domain=75:105] plot ({(2.06)*cos(\x)}, {(2.06)*sin(\x)});

\end{tikzpicture}
\end{figure}

\begin{lemma}\label{lem:leq}
	For $t>0$ and $(z,w) \in \C^2$, 
	$$A_t(h_A(z, w)) \leq A_t(f(z)) \leq \frac{\arctan(e^{-2t})}{\pi}.$$
\end{lemma}
\begin{proof} The first inequality follows from the fact that $h_A(z, w) = f(z)\one_{R_A(z)}(w).$ By our computation of the endpoints, $$g_{-t}\mathcal{T} \subset \left\{re^{i\theta} \in \C: \theta \in \left[\frac{\pi}{2} - \arctan(e^{-2t}), \frac{\pi}{2} + \arctan(e^{-2t})\right]\right\}.$$ Since $r_{\theta} z = e^{i\theta}z$, we have
	$$(A_t h_A)(z, w) \le (A_t f)(z) = \frac{1}{2\pi}|\Theta_t(z)| \leq \frac{\arctan(e^{-2t})}{\pi}.$$
	\end{proof}

\noindent Our next lemma captures the fact that for $(z,w)$ in a set that is only \emph{slightly} smaller than $D_A(e^t/2, e^t)$, we have that $(A_t h_A)(z,w)$ captures a fixed contribution of order $e^{-2t}$. These sets will be used to define the Main term as well as the remaining error terms also shows that the error set $E_t^4$ has finite volume, of which sketches can be seen Figure~\ref{fig:errors}.

\begin{lemma} \label{lem:equality} 
	For $t>0$ if 
	\begin{equation}\label{eq:lemassum}
 (z,w) \in D_A\left(\sqrt{\frac{\cosh(2t)}{2}}, e^t\right) \text{ and }|w| \leq \frac{|z|}{\sqrt{1+e^{-4t}}} 	\end{equation}
	 then
	\begin{equation}\label{eq:fullarc}(A_t h_A)(z,w) = \frac{\arctan(e^{-2t})}{\pi}.\end{equation} Moreover, for all $(z,w)$ such that $(A_t h_A)(z, w) > 0$, \begin{equation}\label{eq:lemmasum2} |w| \le \sqrt{1 + (8A+16A^2) e^{-4t}}|z|.\end{equation}
\end{lemma}
\begin{proof}	
	To show (\ref{eq:fullarc}),  we must show 
	$$|\Theta_t(z,w)| = 2\arctan(e^{-2t}).$$
Note that (see Figure~\ref{fig:arcs}) $$z \in A\left(\sqrt{\frac{\cosh(2t)}{2}}, e^t\right) \Longrightarrow |\Theta_t(z)| = 2\arctan(e^{-2t}).$$ 

\noindent\textbf{Bounds on $w$.} We now consider bounds on $w$. For any $\theta$, $$|r_{\theta}z \wedge r_{\theta}w| = |z \wedge w| \leq A.$$ For $\theta \in \Theta(z)$, we need to verify $r_{\theta} w \in R_A(r_{\theta} z)$. That is, we need to check that $$|\im(r_{\theta} w)| \le |\im(r_{\theta} z)|$$ Since $|w| \leq \frac{|z|}{\sqrt{1+e^{-4t}}}$, 
$$|\im(r_\theta w)| \leq \frac{|z|}{\sqrt{1+e^{-4t}}}.$$ 

\noindent We claim that for any $\theta \in \Theta_t(z)$, $$ \im(r_{\theta} z) \geq \frac{|z|}{\sqrt{1+e^{-4t}}}.$$ Indeed, $\im(r_{\theta} z)$ is minimized over  $\theta \in \Theta_t(z)$ when $\theta = \theta_0$ so that $r_{\theta_0} z$ is on (either) non-horizontal edge of the trapezoid, that is  $$r_{\theta_0} z= |z| e^{i \left(\pi/2 - \arctan\left(e^{-2t}\right)\right)},$$ Note that for $p =u+iv$ on either such edge, $|u| = e^{-2t} v$, so $$|p|= \sqrt{1+e^{-4t}}\im(p).$$ Thus 
$$\im(r_{\theta}z) \geq \im\left(|z|e^{i \left(\pi/2 - \arctan\left(e^{-2t}\right)\right)}\right) = \frac{|z|}{\sqrt{1 + e^{-4t}}}$$ as desired.  To show (\ref{eq:lemmasum2}), consider $\theta$ so that $g_{t} r_{\theta}(z,w) \in R_A(\T)$, that is $$r_{\theta} z = x+iy \in g_{-t} \T.$$ Thus $$e^t/2 \le |z| = |r_{\theta} z|  \le \sqrt{2\cosh(2t)}$$ and $$|x|/y \le e^{-2t} \mbox{ and }  \frac{e^t}{2} \le y \le e^t. $$ Since $r_{\theta} w \in R_A(z),$ 

\begin{align*} |w| =  |r_{\theta} w| & \le \sqrt{\left(|x| + \frac A y \right)^2 + y^2}\\ &= \sqrt{|z|^2 + 2\frac{A|x|}{y} + \frac{A^2}{y^2} }\\ &\le \sqrt{ |z|^2 + 2Ae^{-2t} + 4A^2e^{-2t} }\\ & = |z| \sqrt{ \left(1 + \frac{(2A+4A^2)e^{-2t}}{|z|^2} \right) }\\ &\le |z| \sqrt{\left(1+ 4(2A+4A^2)e^{-4t}\right) }, \end{align*}
where in the last line we are using that $$|z| > e^t/2 \Longrightarrow |z|^{-2} \le 4e^{-2t}.$$
\end{proof}

\subsection{Towards proving Theorem~\ref{thm:error}}
	 To prove  Theorem~\ref{thm:error} we break the left-hand side of \ref{eq:approx} into several terms. First note we have 
	
	\begin{align}\label{eq:sum}\left|N_A^*(\omega, e^t) - \pi e^{2t}\left(A_t \widehat{h_A}\right)(\omega) \right| &= \left| \sum_{(z, w) \in \Lambda_{\omega}^2} \left(\one_{D_A(e^t/2, e^t)}(z,w) -\pi e^{2t} \left(A_t h_A\right)(z, w)\right)\right| \\ \nonumber &= \left|\left(\one_{D_A(e^t/2, e^t)} - \pi e^{2t} A_t h_A\right)^{\SV}(\om)\right|.\end{align} We break (\ref{eq:sum}) into several pieces: a \emph{main term} discussed in \S\ref{sec:main} and four \emph{error terms} in (\S\ref{sec:error}). We will show in \S\ref{sec:upperbound} how to control the error terms. Note that for any $h \in B_c(\C^2), S \subset \C^2$, we can write \begin{align*} \sum_{(z, w) \in \La_{\om} \cap S} h(z, w) &= \sum_{(z,w) \in \La_{\om}} h(z,w)\one_S(z,w) \\ & = (h\cdot \one_S)^{\SV}(\om),\end{align*} 
	
	\subsection{Main term}\label{sec:main} Let $$M_t = \left\{ (z,w) \in D_A\left(\sqrt{\frac{\cosh(2t)}{2}}, e^t\right): |w| < |z|(1+e^{-4t})^{-1/2}\right\}.$$ That is $M_t$ is the collection of pairs satisfying (\ref{eq:lemassum}). Our \emph{main term} will capture the differences of counting pairs in this set and circle averages of $h_A$, that is,
	
	\begin{align}\label{eq:main} m_t(\omega) &=  \sum_{(z, w) \in \La_{\om}^2 \cap M_t} \left(\one_{D_A(e^t/2, e^t)}(z,w) - \pi e^{2t} (A_t h_A)(z, w) \right)  \\ \nonumber &=   \left( \one_{M_t} \cdot \left(\one_{D_A(e^t/2, e^t)} - \pi e^{2t} A_t h_A \right)\right)^{\SV} (\om)\\ \nonumber &=    \left(\one_{M_t} - \pi e^{2t} (A_t h_A)\cdot \one_{M_t} \right)^{\SV} (\om)\end{align} where the last line follows from the fact that $M_t \subset D_A(e^t/2, e^t)$.
	
	\subsection{Error terms}\label{sec:error} We now start to bound the error terms, which arise from pairs in 4 different regions of $\C^2$, See Figure~\ref{fig:errors}.
	
	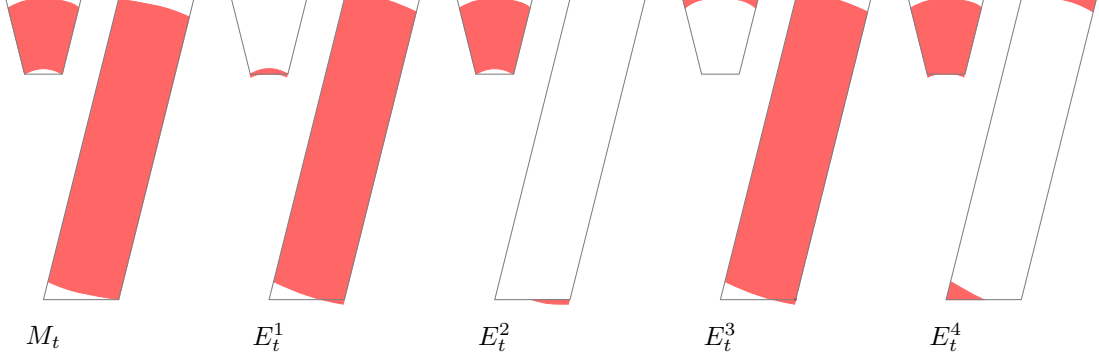
\begin{figure}[tb]\caption{We highlight in red each of the five sets, the main term and four error terms defined in Sections~\ref{sec:main} and \ref{sec:error}. For reference, the trapezoid is $g_{-t}\mathcal{T}$ and the parallelogram to the right of the trapezoid is the set of $w\in \mathcal{R}_A(z)$ for a fixed $z$.}
	\label{fig:errors}
	\centering
\begin{tikzpicture}


\filldraw[red!60] (1/4, 1) -- (.47, 4*.47) to[out=130,in=180] (0,2) to[out=180,in=30] (-.47, 4*.47)--(-1/4, 1) to[out=30,in=150] (1/4,1) ;
\filldraw[red!60] (1,2) to[out=-10,in=155](1.94, 4*.94 - 2)--(1,-2)to[out= 170,in=-25] (1-.94, -1.76);
\node at (0,-2.5) {$M_t$};
\draw[gray](1/4, 1)--(1/2, 2)--(-1/2, 2)--(-1/4, 1)--cycle;
\draw[gray](1,2)--(2,2)--(1,-2)--(0,-2) -- (1,2);
    
\filldraw[red!60] (1/4+3, 1) to[out=150,in=30] (-1/4+3, 1)--(-.24+3, 4*.24) to[out=30,in=180] (3,1) to[out=0,in=150] (.24+3,4*.24);
\filldraw[red!60](1.015+3,4*.015 +2)to[out=0,in=155](1.957+3,4*.957 -2)--(4,-2)--(1-.015+3, -4*.015 -2) to [out =170, in = -25 ](4-.94, -1.76);
\node at (3,-2.5) {$E_t^1$};
\draw[gray](1/4+3, 1)--(1/2+3, 2)--(-1/2+3, 2)--(-1/4+3, 1)--cycle;
\draw[gray](1+3,2)--(2+3,2)--(1+3,-2)--(0+3,-2)--(1+3,2);


\filldraw[red!60] (6+ 1/4, 1) -- (6+ .47, 4*.47) to[out=130,in=180] (6+ 0,2) to[out=180,in=30] (6-.47, 4*.47)--(6-1/4, 1) to[out=30,in=150] (6+1/4,1) ;
\filldraw[red!60](1.015+6,4*.015 +2)to[out=0,in=155](3/2+6,2)-- (7,2);
\filldraw[red!60](1-.015+6, -4*.015 -2) to [out =180, in = -20 ](6.5,-2)-- (7,-2);
\draw[gray](1/4+6, 1)--(1/2+6, 2)--(-1/2+6, 2)--(-1/4+6, 1)--cycle;
\draw[gray](1+6,2)--(2+6,2)--(1+6,-2)--(0+6,-2)--cycle;
\node at (6,-2.5) {$E_t^2$};

\filldraw[red!60] (1/2+9, 2.005)--(9.47, 4*.47) to[out=140,in=0] (9,2.005) --(1/2+9, 2.005);
\filldraw[red!60](-1/2+9, 2) -- (9,2.005) to[out=180,in=40] (9-.47, 4*.47)-- (-1/2+9, 2);
\filldraw[red!60](1.515+8.5,4*.015 +2)to[out=0,in=155](1.957+9,4*.957 -2)--(10,-2)node[red]{$\cdot$}--(1-.015+9, -4*.015 -2) to [out =170, in = -25 ](10-.94, -1.76);
\draw[gray](1/4+9, 1)--(1/2+9, 2)--(-1/2+9, 2)--(-1/4+9, 1)--cycle;
\draw[gray](1+9,2)--(2+9,2)--(1+9,-2)--(0+9,-2)--cycle;
\node at (9,-2.5) {$E_t^3$};

\node at (12,-2.5) {$E_t^4$};
\filldraw[red!60](12+ 1/4, 1) -- (12+ .47, 4*.47) to[out=130,in=180] (12+ 0,2) to[out=180,in=30] (12-.47, 4*.47)--(-.24+12, 4*.24) to[out=30,in=180] (12,1) to[out=0,in=150] (.24+12,4*.24);
\filldraw[red!60](1.957+12,4*.957 -2)to[out=150,in=-10](13.5,2)--(14,2);
\filldraw[red!60](13-.94, -1.76)to[out=-30,in=155](12.5,-2)--(12,-2);
\draw[gray](1/4+12, 1)--(1/2+12, 2)--(-1/2+12, 2)--(-1/4+12, 1)--cycle; 
\draw[gray](1+12,2)--(2+12,2)--(1+12,-2)--(0+12,-2)--cycle; 

\end{tikzpicture}
\end{figure}
	
\subsubsection{Error term 1: Bottom of trapezoid}\label{sec:error1} We define
$$E_t^1 =  D_A\left(e^t/2, \sqrt{\frac{\cosh (2t)}{2}}\right).$$ In particular, for $(z, w) \in E_t^1$ $r_{\theta} z$ hits only the \emph{bottom} of the trapezoid $g_{-t} \T$, and the arc $\Theta_t(z)$ is not the full possible arc of width $2\arctan(e^{-2t})$. Note that the smallest possible length of a vector in $g_{-t} \T$ is $e^t/2$, and for $z = |z|e^{i\varphi}$ with \begin{align*} e^t/2 &< |z| < \sqrt{\frac{\cosh(2t)}{2}} = \sqrt{\frac{e^{2t} +e^{-2t}}{4}} \\ \Theta_t (z) &= (\arcsin(e^t/2|z|)-\varphi, \pi - \arcsin(e^t/2|z|)-\varphi),\end{align*} so $$|\Theta_t(z)|= \pi - 2 \arcsin(e^t/2|z|) = 	2 \arccos(e^t/2|z|).$$ We define

	\begin{align}\label{eq:error1} e_t^1(\omega) &=  \sum_{(z, w) \in \La_{\om}^2 \cap E_t^1} \left(\left(\one_{D_A(e^t/2, e^t)}(z,w) - \pi e^{2t} (A_t h_A)(z, w) \right) \right)  \\ \nonumber &=   \left( \one_{E_t^1} \cdot \left(\one_{D_A(e^t/2, e^t)} - \pi e^{2t} A_t h_A \right)\right)^{\SV} (\om) \\ \nonumber &=   \left(\one_{E_t^1} - \pi e^{2t} (A_t h_A)\cdot \one_{E_t^1} \right)^{\SV} (\om)\end{align} where the last line follows from the fact that $E_t^1 \subset D_A(e^t/2, e^t)$.
	
	\subsubsection{Error term 2: {long $w$ in $D_A$}.}\label{sec:error2} Our second error term consists of pairs $(z, w)$ for which $$|\Theta_t(z)| = 2\arctan(e^{-2t}) \mbox{ but  }|w| > |z|(1+e^{-4t})^{-1/2},$$ so (\ref{eq:lemassum}) is not satisfied. 	That is, $$E_t^2 = \left\{(z, w) \in D_A\left(\sqrt{\frac{\cosh(2t)}{2}}, e^t\right):  |w| > |z|(1+e^{-4t})^{-1/2}\right\},$$ and we define the counting function
	
	\begin{align}\label{eq:error2} e_t^2(\omega) &= \sum_{(z, w) \in \La_{\om}^2 \cap E_t^2} \left(\one_{D_A(e^t/2, e^t)}(z,w) - \pi e^{2t} (A_t h_A)(z, w) \right)   \\ \nonumber &=   \left( \one_{E_t^2} \cdot \left(\one_{D_A(e^t/2, e^t)} - \pi e^{2t} A_t h_A \right)\right)^{\SV} (\om) \\ \nonumber  \end{align} 
	
\subsubsection{Error term 3: The top of the trapezoid}\label{sec:error3} Our third error term is based on the set $$E_t^3 = \{(z,w) \in \C^2: (A_t h_A)(z, w) >0, |z| > e^t\},$$ that is, where $z$ is in the \emph{top} of the trapezoid, and $(z,w) \notin D_A(e^t/2, e^t)$. We set

	\begin{align}\label{eq:error3} e_t^3(\omega) &=  \sum_{(z, w) \in \La_{\om}^2 \cap E_t^3} \left(\one_{D_A(e^t/2, e^t)}(z,w) - \pi e^{2t} (A_t h_A)(z, w) \right)  \\ \nonumber &=   \left( \one_{E_t^3} \cdot \left(\one_{D_A(e^t/2, e^t)} - \pi e^{2t} A_t h_A \right)\right)^{\SV} (\om) \\ \nonumber  &= -\left(\one_{E_t^3} \cdot \pi e^{2t} A_t h_A \right)^{\SV} (\om)\\ \nonumber  \end{align} where the last line follows from the fact that $E_t^3$ is disjoint from $D_A(e^t/2, e^t)$.

\subsubsection{Error term 4: out of $D_A$ with long $w$}\label{sec:error4}

Our fourth and final error term is based on the set where the averaging operator is positive, but $(z, w) \notin D_A(e^t/2, e^t)$. We define $$E_t^4 = \{(z,w) \in \C^2: (A_t h_A)(z, w) >0, e^t/2 \leq |z|\leq e^t, |w|>|z|\},$$ 
\begin{align}\label{eq:error4} e_t^4(\omega) &= \sum_{(z, w) \in \La_{\om} \cap E_t^4} \left(\one_{D_A(e^t/2, e^t)}(z,w) - \pi e^{2t} (A_t h_A)(z, w) \right)   \\ \nonumber &=   \left( \one_{E_t^4} \cdot \left(\one_{D_A(e^t/2, e^t)} - \pi e^{2t} A_t h_A \right)\right)^{\SV} (\om)\\ \nonumber  &= -\left(\one_{E_t^4} \cdot \pi e^{2t} A_t h_A \right)^{\SV} (\om)\\ \nonumber  \end{align} where the last line follows from the fact that $E_t^4$ is disjoint from $D_A(e^t/2, e^t)$.

\subsubsection{Decomposition}\label{sec:decomp} By construction $$D_A(e^t/2, e^t) \cup \{(z,w): (A_t h_A)(z,w)>0\} = M_t \cup \bigcup_{i=1}^{4} E_t^i,$$ and the sets $M_t$ and $E_t^i$ are pairwise disjoint. Therefore

\begin{align}\label{eq:decomp}\left|N_A^*(\omega, e^t) - \pi e^{2t}\left(A_t \widehat{h_A}\right)(\omega) \right| &= \left| \sum_{(z, w) \in \Lambda_{\omega}^2} \left(\one_{D_A(e^t/2, e^t)}(z,w) -\pi e^{2t} \left(A_t h_A\right)(z, w)\right)\right| \\ \nonumber &= \left|\left(\one_{D_A(e^t/2, e^t)} - \pi e^{2t} A_t h_A\right)^{\SV}(\om)\right| \\ \nonumber &= \left| m_t(\om) + \sum_{i=1}^4 e_t^i(\om) \right| \\ \nonumber \end{align}

\section{Upper bounds}\label{sec:upperbound} 

\subsection{Almost sure bounds}\label{sec:almost}

We now show our key almost sure quadratic upper bound for pairs of saddle connections with bounded virtual area (Proposition~\ref{prop:quadratic}), then show how to modify the proof to give Proposition~\ref{prop:Mjbounds}, which controls the main and error terms defined in the previous section.  The main idea is that for each length $e^t$ we will partition the circle into intervals $I_i$ of length $e^{-2t}$ centered at points $\theta_i$.  We will rotate each $I_i$ so that $\theta_i=\pi/2$  and consider $g_t(\omega)$.  We estimate the lengths on $g_t(\omega)$ of  saddle connections whose length on $\omega$ is at  most $e^t$ on $\omega$ and whose  holonomy vector has direction that lies in $\theta_i$.  Then we count saddle connections with these prescribed bounds on $g_t(\omega)$ which will give us the desired bounds on the number of length at most $e^t$ on $\omega$. A major part of this count will be to study the different possibilities for the surfaces $g_t(\omega)$.

\begin{Prop}
\label{prop:Mjbounds} For almost every $\om$, \begin{equation}\label{eq:mainbound}|m_t(\om)| = o(e^{2t})\end{equation} and for $i=1,2, 3,4$ \begin{equation}\label{eq:errorbound}|e_t^i(\om)| = o(e^{2t})\end{equation}
\end{Prop}

\begin{proof}[Proof of Theorem~\ref{thm:error}]To prove Theorem~\ref{thm:error}, we combine (\ref{eq:decomp}) and Prop~\ref{prop:Mjbounds} to get 

\begin{align}\label{eq:decomp:o}\left|N_A^*(\omega, e^t) - \pi e^{2t}\left(A_t \widehat{h_A}\right)(\omega) \right|  &= \left| m_t(\om) + \sum_{i=1}^4 e_t^i(\om) \right| \\ \nonumber &\le |m_t(\om)| + \sum_{i=1}^4 |e_t^i(\om)| \\ \nonumber &= o(e^{2t}) .\\ \nonumber \end{align} 
\end{proof}

\paragraph*{\bf Concluding the proofs} We dedicate the rest of the paper proving the last three results (Proposition~\ref{prop:Mjbounds}, Proposition~\ref{prop:quadratic}, and Lemma~\ref{lem:smallboundary}), which fit together as their proofs all require very similar ingredients. We spend Section~\ref{sec:Notation} through Section~\ref{sec:reducing} establishing preliminary counting results that go into the proofs. In Section~\ref{sec:proof4.2}  we prove Proposition~\ref{prop:quadratic}. We then add one more ingredient giving bounds in the thick part of the stratum in Section~\ref{sec:subquadraticthick}. In Section~\ref{sec:finalproof} we prove Proposition~\ref{prop:Mjbounds}. We conclude with Section~\ref{sec:integralbounds} 
which adds integral bounds in the thin part of the stratum used to prove Lemma~\ref{lem:smallboundary}. 
\subsection{Notation}\label{sec:Notation}

\paragraph*{\bf Systoles} If $\gamma$ is a saddle connection on $\om$, we write $\ell(\gamma)$ for the length of $\gamma$, i.e., $$\ell(\gamma) = |z_{\gamma}|,  \mbox{ where } z_{\gamma} = \int_{\gamma} \omega \mbox{ is the holonomy vector of }\gamma.$$ We write $\ell(\omega) = \ell(\gamma_0)$ for the length of the shortest saddle connection $\gamma_0(\om)$ on $\om$, and $\tilde{\ell}(\om) = \ell(\gamma_1)$ for the length of the shortest saddle connection $\gamma_1(\om)$ not homologous to the shortest saddle connection $\gamma_0(\om)$. We define $$BC= \{\omega \in \hh: \gamma_0(\om) \mbox{ bounds a cylinder}\}$$ and given $\epsilon$,  $$BC_{\epsilon} = \{\omega \in \hh: \gamma_0(\om) \mbox{ bounds a cylinder of width } \le \epsilon \}.$$ 

\paragraph*{\bf Scales} Fix $0 < \sigma < 1$. Let $$s_{\sigma}(\om)  = \frac{\log \ell(\om)}{\log \sigma}, \tilde{s}_{\sigma}(\om) = \frac{\log \tilde{\ell}(\om)}{\log \sigma},$$ so $$\sigma^{s_{\sigma}(\om)} = \ell(\om), \sigma^{\tilde{s}_{\sigma}(\om)} = \tilde{\ell}(\om).$$

\paragraph*{\bf Angles} On a base surface $\omega$ we refer to a holonomy vector $z$ of a saddle connection $\gamma$ without subscripts. On the surface $g_t r_{\theta}\omega$ the image holonomy vector $g_t r_\theta z$ will be denoted $z_{\theta, t}$. If $z = |z|e^{i\varphi}$, we define $\theta_z = \pi/2 - \varphi$ to be the angle so that $r_{\theta_z} z$ is vertical, that is $r_{\theta_z} = |z|i$.

\paragraph*{\bf Triangulations}
Let $N = N(\hh)$ be the maximum number of edges in a triangulation by saddle connections of any $\omega \in \hh$. This number exists since saddle connections joining a pair of zeroes lie in different homotopy classes and there is a finite number of homotopy classes of curves that are disjoint except at common vertices which are the zeroes of $\omega$. 

\paragraph*{\bf The set} Define $$P_A(\omega,e^t)= \La_{\om}^2 \cap D_A(e^t/2, e^t),$$ so $|P_A(\omega, e^t)| = N_A^*(\omega,e^t).$

\paragraph*{\bf Partitioning the circle} For each $t>0$ partition $[0,2\pi)$ into $\lfloor e^{2t} \rfloor$ intervals $I(\theta_i)$ of radius $\frac{\pi}{\lfloor e^{2t} \rfloor}$ centered at points $\theta_i$ for $i=1,\ldots, \lfloor e^{2t} \rfloor$.  We will look at counting on the finite set of surfaces $\{\omega_i = g_{t} r_{\theta_i}\omega\}$ and then prove Lemma~\ref{lem:expand} to observe that in each of these intervals centered at $\theta_i$ lengths change by at most a multiplicative constant which will be absorbed in our estimates.

\paragraph*{\bf Ratios}For each $i$, define $$j = j(i) = \left\lfloor s_{\sigma}(\om_i) - \tilde{s}_{\sigma}(\om_i)\right\rfloor,$$ so 
$$\sigma^{j+1} < \frac{\ell(\om_i)}{\tilde{\ell}(\omega_i)} \leq \sigma^j.$$

\subsubsection{Measure bounds}\label{sec:measurebounds}

Next we state a result originally due to Masur-Smillie \cite{MasurSmillie91}, estimating the measure of the set of surfaces with two non-homologous short saddle connections. \begin{lemma}\label{lem:MS}\cite[Equation 7]{MasurSmillie91}
	For all $\epsilon, \kappa >0$, the Masur-Smillie-Veech measure of the set $V_1(\epsilon, \kappa) \subset \mathcal{H}$ of $\omega$ which have a saddle connection of length at most $\epsilon$, and a non-homologous saddle connection with length at most $\kappa$ is $O(\epsilon^2 \kappa^2)$.
\end{lemma}

\subsubsection{Counting lemmas}\label{sec:countinglemma} Finally we state a counting lemma which is summarized from \cite[\S 3.6.2 and 3.6.3]{AthreyaCheungMasur19}. Recall $N(\omega,\epsilon_0)$ is the number of nonhomologous saddle connections of length at most $\epsilon_0$.
\begin{lemma}\label{lem:ACM} Fix $\omega \in \hh$ with shortest saddle connection $\gamma$ and second shortest saddle connection $\gamma'$.  If either $\gamma$ does not bound a cylinder, or $\gamma$ bounds a cylinder of width at least $\epsilon_0$ then $$N(\omega, \epsilon_0)^2  =O(|\ell(\gamma')|^{-2N}).$$

\end{lemma}

\subsection{Counting bounds with systoles}\label{sec:boundswsystoles} The following is a slight modification of  ~\cite[Theorem~5.1]{EskinMasur01}.

\begin{lemma}
\label{lem:EM}
For any $L_0>0$  and $\delta>0$   there exists $C=C(\delta,L_0)$ such that for any $L < L_0$ and any surface $(X,\omega)$ in the stratum 
\begin{equation}\label{eq:less optimal count}
N(\omega, L)\leq C\left(\frac{L}	{\ell(\omega)}\right)^{1+\delta}
\end{equation}

\end{lemma}
\begin{proof} Fix $L_0 >0$, $\delta >0$, let $L<L_0$, and let $\omega \in \mathcal{H}$.   \cite[Theorem 5.1]{EskinMasur01} states that there is a $\kappa = \kappa(\hh) >0$ and $C' = C'(\delta, \mathcal{H})$ so that for $L< \kappa$,
$$N(\omega, L) \leq C' \left(\frac{L}{\ell(\omega)}\right)^{1+\delta}.$$
Thus, if $L_0 < \kappa$ we are done. We now consider the case $L_0 > \kappa$. Divide $[0,2\pi)$ into equally sized intervals $J_i$ of radius $\frac{\kappa^2}{4L_0^2}$ and let $\phi_i$ be the center of $J_i$. Note that there are $O(L_0^2)$ such intervals. For any $z \in \Lambda_{\omega} (L)$ choose $\phi_i$ with angle $\theta_z\in J_i$. 

\noindent\textbf{Rotating to (almost) vertical} We rotate $z$ to almost vertical via $r_{-\phi_i}$, and then shrink $r_{-\phi_i} z$ to have length less than $\kappa$ by applying $g_t r_{-\phi_i}$ where $e^{t}=\frac{2L_0}{\kappa}$, which we justify as follows. The largest possible imaginary component for $ r_{-\phi_i}z$ is bounded above by $|z| <L<L_0$, so
$$\im(z_{-\phi_i, t}) = \frac{\kappa}{2L_0} \im(r_{-\phi_i}z) \leq \frac{\kappa}{2}.$$
After rotating by $-\phi_i$, $r_{-\phi_i}z$ must lie in $B(0,L)$ with angle with the vertical in $\left(-\frac{\kappa^2}{4L_0^2}, \frac{\kappa^2}{4L_0^2}\right),$ so the real component of $r_{-\phi_i}z$ must satisfy
$$\re(z_{-\phi_i, t}) = \frac{2L_0}{\kappa}\re(r_{-\phi_i}z) \leq \frac{2L_0}{\kappa} \im(r_{-\phi_i}z) \tan\left(\frac{\kappa^2}{4L_0^2}\right) \leq \frac{2 L_0^2}{\kappa} \frac{\kappa^2}{4L_0^2} = \frac{\kappa}{2},$$
where we used $\tan(x) < x$ for $0<x<\frac{1}{4}$.
Thus 
$$|g_tr_{-\phi_i} z| \leq \kappa.$$

\noindent Moreover, for each $i$ the systoles satisfy
$$\ell(\omega) \leq \ell(g_tr_{\phi_i}\omega) \cdot  \frac{2L_0}{\kappa}.$$
So 
$$N(\omega, L)  = \displaystyle\sum_{i=1}^{O(L_0^2)} C'\left(\frac{L}{\ell(g_tr_{-\phi_i} \omega)}\right)^{1+\delta} \leq O(L_0^2) \left(\frac{2L_0}{\kappa}\right)^{1+\delta}C'\left(\frac{L}{\ell(\omega)}\right)^{1+\delta} .$$ Combining the terms $O(L_0^2)\left(\frac{2L_0}{\kappa}\right)^{1+\delta}C'$, we get a constant $C = C(\delta, L_0)$ as desired.

\end{proof}

\subsection{Geodesic flow length bounds for pairs}\label{sec:geodesicboundspairs}
 In this section we collect bounds that hold under certain assumptions for pairs of holonomy vectors.
 We begin by fixing some additional constants associated to $A$ which we will use in our proofs. Set  $$\rho_A = (8A +16A^2), \widehat{\rho}_A = \sqrt{2}\sqrt{1+\rho_A} + 2A, \  \epsilon_0=8\pi\widehat{\rho}_A.$$
 
\noindent Note that, by (\ref{eq:lemmasum2}), if $\left(A_t h_A\right)(z, w) >0$, then $$|w| \leq \sqrt{1+ \rho_A e^{-4t}}|z|. $$We define the \emph{large set} (Figure~\ref{fig:large})
$$L_A(e^t) = \left\{(z,w): \frac{e^t}{2} \leq |z| \leq \sqrt{2\cosh(2t)}, |w\wedge z|\leq A,   |w| \leq \sqrt{1+ \rho_A e^{-4t}}|z| \right\},$$
which satisfies, for $1 \le i \le 4$,
$$E_t^{i} \subseteq L_A(e^t) \quad \text{ and } P_A(\omega, e^t) \subseteq L_A(e^t) \cap \La_{\om}^2.$$ Recall that for $z \in \C$, $\theta_z = \pi/2 - \arg(z)$ is the angle so that $r_{\theta_z} z$ is vertical.

	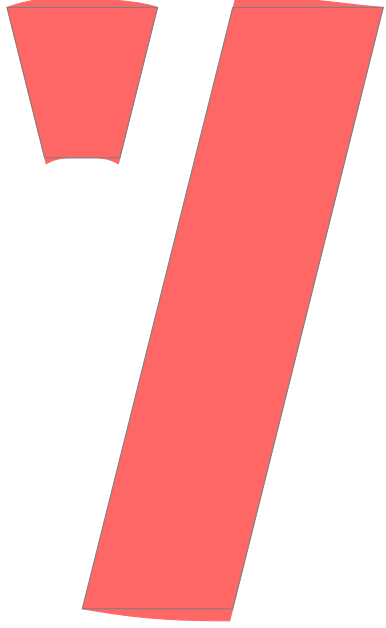
\begin{figure}[tb]\caption{Here we given an image depicting the large set $L_A(e^t)$ in red. The trapezoid is the set of $z \in g_{-t}\mathcal{T}$ and the parallelogram is the set of $w\in \mathcal{R}_A(z)$ for a fixed $z$. The set of $z$ in $L_A(z)$ contains both the trapezoid and red points outside.}
	\label{fig:large}
	\centering
\begin{tikzpicture}[scale = 2.0]

\filldraw[red!60] (.24, 4*.24) -- (.5, 2) to[out=160,in=180] (0,2.06) to[out=180,in=20] (-.5, 2)--(-.24, 4*.24) to[out=30,in=180] (0,1) to[out=0,in=150] (.24,4*.24) ;
\filldraw[red!60] (1.02,4*1.02-2) to[out=0,in=175](2, 2)--(1-.02, -4*.02 -2)to[out= 180,in=-10] (0, -2);
\draw[gray](1/4, 1)--(1/2, 2)--(-1/2, 2)--(-1/4, 1)--cycle; 
\draw[gray](1,2)--(2,2)--(1,-2)--(0,-2)--(1,2);
\node at (0,-2.5) {$L_A(e^{t})$};

\end{tikzpicture}
\end{figure}

 \begin{lemma}\label{lem:pairs} For all $(z, w) \in L_A(e^t)$
 $$\frac{1}{2} \leq |g_t r_{\theta_z} z| \leq \sqrt{2},$$ and
 	$$|g_t r_{\theta_z} w| \leq \widehat{\rho}_A .$$
 \end{lemma}
 \begin{proof} By definition $r_{\theta_z} z = |z| i$. Since  $$\frac{e^t}{2} \leq |z| \leq e^t\sqrt{1+e^{-4t}},$$ we have $$|g_{t} r_{\theta_z} z| = |e^{-t} |z| i| = e^{-t} |z|,$$ so $$\frac{1}{2}\leq |g_tr_{\theta_z} z| \leq \sqrt{1+e^{-4t}} \leq \sqrt{2}.$$
By Lemma~\ref{lem:equality}, $|w| \leq |z|\sqrt{1+\rho_Ae^{-4t}}$, so the vertical component of $r_{\theta_z} w$ is at most $|z|\sqrt{1+\rho_A e^{-4t}}$, so $$\im(g_tr_{\theta_z}w) \leq |g_tr_{\theta_z}z|\sqrt{1+\rho_A e^{-4t}}.$$ Since $|z \wedge w| \le A$,  $|\re(g_tr_{\theta_z} w)|\cdot |g_tr_{\theta_z} z| \leq A.$ Hence
 	$$|g_tr_{\theta_z} w| \leq |\im(g_tr_{\theta_z} w)| + |\re(g_tr_{\theta_z} w)| \leq |g_tr_{\theta_z} z| \sqrt{1+\rho_A e^{-4t}} + \frac{A}{|g_tr_{\theta_z} z|} \leq \sqrt{2}\sqrt{1+\rho_A} + 2A=\widehat{\rho}_A  .$$ 
 \end{proof}

\subsection{Reducing to finitely many surfaces}\label{sec:reducing}
This next lemma shows that within a fixed range of angles $\theta$ and times  $t$, the magnitude $|g_tr_\theta z |$ cannot change more than an explicitly quantifiable amount.
\begin{lemma}
\label{lem:expand}
Given $t>0$, $\theta_0 \in [0, 2\pi)$, define $$I_t(\theta_0) = \left(\theta_0 - \pi e^{-2t}, \theta_0 + \pi e^{-2t}\right).$$ Then for any $t > 0$, any $z \in \C$, $\theta\in I(\theta_0) $, and $s \in [t, t+ \log 2)$  we have $$\frac{1}{8\pi} \leq \frac{|g_s r_\theta z|}{|g_t r_{\theta_0} z|}\leq 8\pi.$$ 
\end{lemma}

 \begin{proof} First write $$g_s r_{\theta} z = g_s r_{\theta- \theta_0} g_{-t} (g_t r_{\theta_0} z).$$ Thus it suffices to control the operator norm of $g_s r_{\psi} g_{-t}$ where $|\psi| = |\theta-\theta_0| \le \pi e^{-2t}.$ We proceed to do that. We note that $g_s r_{\psi} g_{-t} = g_s r_{\psi} g_{-s} g_{s-t}$. Note that $\|g_{s-t}\|_{op} \le 2$, and $$g_s r_{\psi} g_{-s} = \begin{pmatrix} \cos \psi & -e^{2s} \sin \psi \\ e^{-2s} \sin \psi & \cos \psi\\ \end{pmatrix}.$$ The operator norm of a matrix $A$ is the square root of the largest eigenvalue of $A^{T}A$. In this case $$(g_s r_{\psi} g_{-s})^T(g_s r_{\psi} g_{-s} ) = \begin{pmatrix} \cos^2 \psi  + e^{-4s} \sin^2 \psi& -\sin(2\psi) \sinh(2s)  \\ -\sin(2 \psi)\sinh(2s)  & \cos^2 \psi + e^{4s} \sin^2(\psi) \\ \end{pmatrix}.$$
 For a determinant $1$ matrix $M$, with $|\tr(M)|>2$ the largest eigenvalue $\lambda^+(M)$ is given by $$\lambda^+(M) = \frac{\tr(M)}{2} +\sqrt{\left(\frac{\tr(M)}{2}\right)^2 -1}.$$ Applying this to $M = (g_s r_{\psi} g_{-s})^T(g_s r_{\psi} g_{-s} )$, we have $$\frac{\tr(M)}{2} = \cos^2 \psi + \cosh(4s) \sin^2\psi = 1 + (\cosh(4s)-1)\sin^2\psi = 1+ 2\sinh^2(2s)\sin^2\psi,$$ so $$\lambda^+(B)  = 1 + 2 \sinh^2(2s) \sin^2\psi + \sqrt{2} \sinh(2s)\sin\psi.$$
Since $|\psi| < \pi e^{-2t}$, $| \sin\psi| < \pi e^{-2t}$, and $\sinh(2s) < 2e^{2t},$ so \begin{align*} \lambda^+(M) &\le 1+2\pi^2e^{-4t}e^{4t} + \sqrt{2}2e^{2t} \pi e^{-2t} \\ &\le 1 + 8\pi^2 + 2 \sqrt 2 \pi. \end{align*} Therefore, $$\|g_s r_{\psi} g_{-s}\|_{op} = \sqrt{\lambda^+(M) } \le \sqrt{1+8\pi^2 + 2 \sqrt 2 \pi} \le 4\pi,$$ and so $$\|g_s r_{\psi} g_{-t} \|_{op} \le \|g_s r_{\psi} g_{-s}\|_{op} \|g_{s-t}\|_{op} \le 8\pi.$$
\end{proof}

\subsection{Proof of Proposition~\ref{prop:quadratic}} \label{sec:proof4.2}
Let $$\mathcal{H}_\epsilon = \{ \omega \in \hh: \ell(\omega) \leq \epsilon\}$$ denote the $\epsilon$-thin part of the stratum. Before proceeding with the proof, we define three families of sets exhausting the thin part $\hh_{\epsilon}$ of the stratum, in terms of the length of the shortest and second shortest non-homologous saddle connections.  Recall we have fixed a scale parameter $0<\sigma <1$.  For each integer $j>0$ let $F(j)$ be the set of $\om$ with a shortest saddle connection $\gamma_0(\om)$ with length between $\sigma^{j+1}$ and $\sigma^j$ and a non-homologous saddle connection of length at most $\sigma^{\frac{j}{4N}}$. Recalling that  $$s_{\sigma}(\om)  = \frac{\log \ell(\om)}{\log \sigma}, \tilde{s}_{\sigma}(\om) = \frac{\log \tilde{\ell}(\om)}{\log \sigma},$$ we have \begin{equation}\label{eq:thindef} F\left(j\right) = \left\{\omega \in \hh: \lfloor s_{\sigma}(\om) \rfloor = j, \tilde{s}_{\sigma}(\omega) \geq \frac {j}{4N} \right\} .\end{equation}

\noindent Next let
$G(j)$  be the set of $\om$ where the shortest saddle connection $\gamma_0(\om)$ has length between $\sigma^{j+1}$ and  $\sigma^j$, the shortest  non-homologous saddle connection $\gamma_1(\om)$ has length at least $\sigma^{\frac{j}{4N}}$, and $\gamma_0(\om)$ either does not bound a cylinder or bounds a cylinder of width at least $\epsilon_0$. That is,
\begin{equation}\label{eq:thindefG} \nonumber G(j) = \left\{\omega \notin BC_{\epsilon_0}: \lfloor s_{\sigma}(\om) \rfloor = j, \tilde{s}_{\sigma}(\omega) < \frac {j}{4N}  \text{ and } \gamma_0(\om) \text{ bounds no cylinder or a cylinder width }>\epsilon_0 \right\} .\end{equation}

 \noindent Finally let $H(j)$ be the set of $\om$ where the shortest saddle connection $\gamma_0(\om)$ has length between $\sigma^{j+1}$ and $\sigma^j$ and bounds a cylinder of width at most $\epsilon_0$, $$H(j) = \{ \omega \in BC_{\epsilon_0}: \lfloor s_{\sigma}(\om) \rfloor = j \text{ and } \gamma_0(\om) \text{bounds a cylinder of width at most }\leq\epsilon_0\}. $$ 
 
\noindent Now choose $j_0$ so that  $\sigma^{j_0+1}\leq \epsilon<\sigma^{j_0}$. Thus $$\mathcal{H}_\epsilon \subset \bigcup_{j\geq j_0}F(j)\cup G(j)\cup H(j).$$ 

\begin{proof}

\noindent\textbf{Choosing parameters.} We begin by  fixing  a small $\epsilon>0$. Choose $\kappa<1$ and $\delta < \frac{1}{4N}$ small enough that $\kappa+\delta(1+\kappa) < \frac{1}{4N}$. Fix $L_0>0$.

\noindent\textbf{Breaking up the set of intervals.} We have partitioned the circle into $e^{2t}$ intervals centered at angles $\theta_i$. 
We will break up the set indices $i$ into four  subsets depending on the part of $\mathcal{H}$ the surface  $g_t r_{\theta_i}\omega$ lies in. For each of the four  sets we will count the number of intervals $I_i$ in that set and
for $z$ the holonomy vector that is the first element in a   pair  $(z,w)$ we consider the interval $I_i$ so that $\theta_z\in I_i$. Then for that $I_i$ we count the number of pairs $(z,w)$ with $z$ determining $I_i$.    Multiplying by the number of $I_i$ will give us the desired bound.

\paragraph*{\bf{Thick Case}}
Let $i$ be an index such that $g_tr_{\theta_i}\omega\in \mathcal{H}\setminus \mathcal{H}_\epsilon$. That is, the surface $g_t r_{\theta_i} \om$ is $\epsilon$-thick.  If $\beta$ is a saddle connection with holonomy vector $z = |z| e^{2\pi i \theta_z}$ so that $|z|\leq e^t$ and $\theta_z\in I_i$, then by the second conclusion of Lemma~\ref{lem:pairs}, on $g_tr_{\theta_i} \omega$, the length of $\beta$ is at most $\widehat{\rho_A}$. The same bound holds  for any $\gamma$ with holonomy $w$ satisfying $|z\wedge w|\leq A$ and of length at most $e^t$. The number of saddle connections of length at most $\widehat{\rho}_A$ on a surface in the thick part is  bounded in terms of $\epsilon$ and $\widehat{\rho}_A$. This is an easy consequence for example of the compactness of the thick part. The bound for the number of pairs is the square of this quantity which of course is still uniformly bounded.   Thus there are  $O(1)$ pairs of saddle connections with the above property. Since there are $O(e^{2t})$ such intervals $I_i$ we have the desired bound. 

\paragraph{\bf{Thin Case}}
We now turn to the much more complicated proof in the case that the surface $g_tr_{\theta_i}\omega\in \hh_\epsilon$, that is, is in the thin part.

\begin{rmk} We remark that at this point in the published version \cite{AFM23} on page 39, we treated each case $F(j)$, $G(j)$ and $H(j)$ separately. Along the way in the case of $F(j)$, we used Theorem~\ref{thm:Nevo}, but we did not convey in the published version that the choice of $t_0$ depends on $j$, invalidating \cite[Equation 6.7]{AFM23} which required $t_0$ independent of $j$. However, we are able to fix our proof following a suggestion of Jon Chaika by combining the cases into a single summatory function before applying Nevo's theorem~\ref{thm:Nevo}. \end{rmk}

\paragraph*{\bf Counting pairs in the thin set} Following the notation from Section~\ref{sec:Notation}, for each $j$,  define the number of pairs $N_j^*$ by
\begin{align*} N_j^F(\omega) &= \#\left\{(z,w)\in L_A(e^t) \cap \La_{\om}^2: \exists i \text{ so that } \theta_z \in I(\theta_i), g_tr_{\theta_i}\omega \in F(j)\right\} \\  N_j^G(\omega) &= \#\left\{(z,w)\in L_A(e^t) \cap \La_{\om}^2: \exists i \text{ so that } \theta_z \in I(\theta_i), g_tr_{\theta_i}\omega \in G(j)\right\} \\ N_j^H(\omega) &= \#\left\{(z,w)\in L_A(e^t) \cap \La_{\om}^2: \exists i \text{ so that } \theta_z \in I(\theta_i), g_tr_{\theta_i}\omega \in H(j)\right\} .\\\end{align*}

\paragraph*{\bf Scale constants} For each positive integer $j$ let $$c_j^H= c_j^F = \sigma^{-j(2+ 2\delta)}  \mbox{ and } c_j^G=\sigma^{-j}.$$ Then we have

\begin{lemma}
\label{lem:upperbound}
For each $j$ and $i$, if $g_t r_{\theta_i} \omega \in *(j)$, then
$\#\left\{(z,w)\in L_A(e^t) \cap \La_{\om}^2: \theta_z \in I(\theta_i)\right\} \ =O(c_j^*)$

\end{lemma}
\begin{proof}

We start with the case of $F(j)$ and $H(j)$, which we can treat similarly. Consider any $(z,w)\in L_A(e^t)$ with the property that $\theta_z \in I(\theta_i).$ 
Lemma~\ref{lem:pairs} and Lemma~\ref{lem:expand} show that $|g_tr_{\theta_i}z| \leq 8\pi\sqrt{2}$ and $|g_tr_{\theta_i}w| \leq 8\pi\widehat{\rho}_A.$ Apply Lemma~\ref{lem:EM} with $L_0 = 8\pi\widehat{\rho}_A$. Thus for each $j$ and $i$, if $g_t r_{\theta_i} \omega \in F(j)$ or $g_t r_{\theta_i} \omega \in H(j)$,
\begin{equation}\label{eq:FNjbound}
\#\left\{(z,w)\in L_A(e^t) \cap \La_{\om}^2: \theta_z \in I(\theta_i)\right\} \leq N(g_{t}r_{\theta_i}\omega,L_0)^2= O\left(\sigma^{-j(2+ 2\delta)}\right)=O(c_j^F) = O(c_j^H),
\end{equation}
where the constants in the big $O$ notation depend only on $\delta$, $A$, and in fact are independent of the interval index $i$ and the surface $\omega$. 

We now consider $G(j)$. The second shortest saddle connection on $g_tr_{\theta_i}\omega$ has length at  least $\sigma^{\frac{j}{4N}}$ by definition. Lemma~\ref{lem:ACM} gives the number of saddle connections of length to be at most $\epsilon_0 = 8\pi\widehat{\rho}_A$ is $O((\sigma^{\frac{j}{4N}})^{-2N})=O(\sigma^{-j/2}).$  Thus the total number of pairs is bounded above by $O(\sigma^{-j})=O(c_j^G)$, where the implicit constants depend only on the parameters $\sigma$ and $j$.
\end{proof}

\paragraph*{\bf Building functions} 

For each $j$, define the function $\psi_j: \mathcal{H}\to \mathbb{R}$  by $\psi_j = c_j^F \one_{F(j)} + c_j^G \one_{G(j)} + c_j^H \one_{H(j)}$. We then consider the sum $$\psi = \sum_{j=1}^\infty \psi_j.$$ 
There is a corresponding function $\phi$ using $N_j^*$ in place of $c_j^*$ which records the actual count.  Lemma~\ref{lem:upperbound}  says that $\phi(\omega)\leq \psi(\omega)$.
 Assume the circle averages $A_t\psi$ converge as $t\to\infty$.   Since there are $e^{2t}$ intervals, the assumption together with Lemma~\ref{lem:upperbound} says that an upper bound on the total number of pairs which land in the thin part  
for a given surface $\omega$ is  $O(e^{2t}A_t\psi)$ for large enough $t$. We now prove the assumption using Theorem~\ref{thm:Nevo}.

\paragraph*{\bf $K$-finiteness and integrability} Since each $F_j$, $G_j$, and $H_j$ is rotation invariant, the function $\psi$ is $K$-finite. We claim $\psi \in L^{1+\kappa}$ for some $\kappa > 0$. Since the support of the $\psi_j$ are disjoint, $\psi^{1+\kappa} = \sum_{j=1}^\infty \psi_j^{1+\kappa}.$ Using the fact that the functions are nonnegative we can interchange limit and integral to compute
\begin{equation}\label{eq:psidecomp}
    \int_{\mathcal{H}} \psi^{1+\kappa} \,d\mu = \sum_{j=1}^\infty \int_{\mathcal{H}} \psi_j^{1+\kappa}
    \,d\mu = \sum_{j=1}^\infty (c_j^F)^{1+\kappa}\mu(F(j)) + (c_j^G)^{1+\kappa} \mu(G(j)) + (c_j^H)^{1+\kappa} \mu(H(j)).\end{equation}
We now address each term of the right hand side of \eqref{eq:psidecomp} separately. Similar to our definition of the set $F(j)$ define  $E(j)$ to be the set of $\omega$ with a saddle connection of length between $(8\pi)^{-1} \sigma^{j+1}$ and  $8\pi \sigma^j$,  and a nonhomologous saddle connection of length at most $8\pi \sigma^\frac{j}{N}.$ By Lemma~\ref{lem:MS},
\begin{equation}\label{eq:Fmeasurebound} \mu(F(j)) \leq \mu(E(j)) = O\left(\sigma^{2j + 2\frac{j}{4N}}\right) = O\left(\sigma^{2j} \sigma^{\frac{j}{2N}}\right).
\end{equation}
To bound the measure of $G$, we use the estimate that the shortest saddle connection has size $\sigma^j$, so $\mu(G(j)) \leq \sigma^{2j}$ by Lemma~\ref{lem:MS}. 
Notice that in $H$ the shortest saddle connection has size $\sigma^{j}$, and since it bounds a cylinder of width at most $\epsilon_0$,  by \cite[Proof of Theorem 10.3]{MasurSmillie91}, we have
$$\mu(H(\sigma^j))= O\left(\sigma^{3j}\right).$$

\paragraph*{\bf Integrability} Applying the estimates from each case to \eqref{eq:psidecomp}, the definition of the constants $c_j^*$  and. our choices of $\kappa$ and $\delta$ imply
\begin{align}
   & \int_{\mathcal{H}} \psi^{1+\kappa} \,d\mu =  O\left(\sum_{j=1}^\infty(c_j^F)^{1+\kappa}\mu(F_j)+(c_j^G)^{1+\kappa}\mu(G_j)+(c_j^H)^{1+\kappa}\mu(H_j)\right) \nonumber \\ &=O\left( \sum_{j=1}^\infty \sigma^{-j(2\delta)(1+\kappa){-2j\kappa}} \sigma^{\frac{j}{2N}} + \sigma^{-j(1+\kappa)} \sigma^{2j} + \sigma^{-j(2\delta)(1+\kappa) {-2j\kappa}} \sigma^{j}\right) < \infty. \label{eq:geometric_series}\end{align} Now we are in a situation where we can apply Theorem~\ref{thm:Nevo}. Choose a mollifier $\eta \in C_0^\infty(\mathbb{R})$ such that $\eta(t)=1$ on $[-\log(2), 0]$ and $\int\eta(t)dt=1$ so that by Theorem~\ref{thm:Nevo}, for almost 
every $\omega\in \mathbb{H}$, there is some $t_0$ large enough depending on $\omega$  so that for $t>t_0$
\begin{equation}\label{eq:new_upper_bounds}\int_{t}^{
t+\log(2)}e^{2t}(A_s \psi)(\omega) ds \leq e^{2t}\int_{-\infty}^\infty \eta(t-s) (A_s \psi )(\om)ds \leq 3e^{2t}\int_{\mathcal \hh} \psi d\mu = 3e^{2t} \sum_{j=1}^\infty \int_{\mathcal{H}} \psi_j \,d\mu = O(e^{2t}).\end{equation}
Switching the order of integration

\begin{equation}
\label{int:bound}
e^{2t}\sum_{i=1}^{e^{2t}} \int_{I(\theta_i)} \int_t^{t+\log(2)} \psi(g_sr_\theta \omega)\,ds\,d\theta=e^{2t}\int_{t}^{
t+\log(2)}(A_s\psi)(\omega) ds=O(e^{2t}) \end{equation}

\paragraph*{\bf Fattening} 
Now at this point, we replace $\psi$ with a function defined by constants that are a fixed multiple $8\pi$ of $c_j^F, c_j^H,$ and $c_j^G$ on the respective domains of $F, G, H$.  
Just as we considered $E(j)$ related to $F(j)$, we can do the same  type of fattening to guarantee that by  fattening the domains of $F, G, H$, we then have a function that is constant on each interval $I(\theta_i)$. For example in the case of $F(j)$, if there is some $i \in \{1, \ldots, e^{2t}\}$ so that $g_t r_{\theta_i}(X,\omega) \in F(j)$, then Lemma~\ref{lem:expand} guarantees that $g_sr_\theta \omega \in E(j)$ whenever 
$$(\theta, s) \in I(\theta_i) \times [t, t+\log(2)].$$
Then \eqref{int:bound} and the fact that $\phi\leq \psi$ gives the desired  upper bound 
$$e^{2t}A_t(\phi) (\log(2)) \leq \int_t^{t+\log(2)} e^{2t} A_s(\psi)(\omega)\,ds = O(e^{2t}).$$
This finishes the proof of  Proposition~\ref{prop:quadratic}.
\end{proof}

\color{black}

\noindent The  next corollary is a small modification of Proposition~\ref{prop:quadratic}, which for a given $\widehat{\epsilon}$ says \begin{Cor}\label{cor:thin} Fix notation as above. Then, {for almost every $\omega \in \mathcal H$} the count of pairs in the thin part is bounded as follows: {For all $t$ sufficiently large, we have}

	$$\#\left\{(z,w)\in L_A(e^t) \cap \La_{\om}^2:\exists\, i \text{ so that } \theta_z \in I(\theta_i),\, g_tr_{\theta_i}\omega \in \mathcal{H}_{\widehat \epsilon}\right\}
	=O(\widehat{\epsilon}^{\frac{1}{2N} - 2\delta}e^{2t}).$$ 
Here, the implied constant is independent of $\omega$.
\end{Cor}

\begin{proof}

	We follow exactly the proof of  Proposition~\ref{prop:quadratic}  where we have the additional assumption now that $\sigma^j \leq \widehat{\epsilon}$. This comes into effect in \eqref{eq:geometric_series}
when we restrict  $\sigma^j<\widehat\epsilon$ which limits the set of possible $j$ to $j\geq \frac{\log(\widehat{\epsilon})}{\log(\sigma)}.$ So then the computation (bounding the geometric series by a multiple of the first term) from \eqref{eq:geometric_series} implemented in \eqref{eq:new_upper_bounds} yields an upper bound of
\begin{equation*}
 O\left(e^{2t} \sum_{j=\frac{\log\widehat{\epsilon}}{\log \sigma}}^\infty \sigma^{j\left(\frac{1}{2N} - 2\delta\right)}  \right)  =  O(\widehat{\epsilon}^{\frac{1}{2N}-2\delta}e^{2t}) 
\end{equation*}
\end{proof}
\color{black}

\subsection{Subquadratic decay in the thick part of the stratum} \label{sec:subquadraticthick}

As a complement to Corollary~\ref{cor:thin} which gives sufficient bounds in the thin part of the stratum, we want to understand decay of pairs of saddle connections in the thick part of the stratum, that is the set of surfaces where the systole has length at least $\epsilon$. In this section we will give an estimate on the measure of a subset of the thick part of the stratum which correspond to each of the four error terms introduced in Section~\ref{sec:error}. We prove the error terms satisfy the definitions given in this section in the proof of Proposition~\ref{prop:Mjbounds}.

To this end we will prove the following  lemma which gives measure bounds for pairs of saddle connections near the boundary of $R_A(\mathcal{T})$. Given $L, \widehat\epsilon,\epsilon'>0$, and $L'\in\{\frac{1}{2},1\}$ define $\Omega(\widehat\epsilon,\epsilon',L,L')$ to be the set  of surfaces $\omega$ such 
that $\omega$ is $\widehat\epsilon$-thick, and
there are $(z, w) \in \La_{\om}^2 \cap B(0, L)^2$ where at least one of the following holds:
\begin{enumerate}
\item Ratio of vertical components is bounded: $$1-\epsilon'\leq \frac{|\im  w|}{|\im z|}\leq 1+\epsilon'$$ 
\item Area is close to $A$ $$(1-\epsilon')A\leq |z\wedge w|\leq (1+\epsilon')A$$
\item Vertical component of $z$ is close to $L'$:
$$|\im (z) -L'|<\epsilon'$$
\item  Vertical and horizontal components are close: $$(1-\epsilon')\im z\leq |\re z|\leq (1+\epsilon')\im z.$$

\end{enumerate}

\begin{lemma}
\label{lem:bad}
Let $D$ be the complex dimension of the stratum. For each $L$, there exists $C$ so that for all $\widehat\epsilon,\epsilon'$  $$\mu\left(\Omega(\widehat\epsilon,\epsilon',L,L')\right) \le C\left(\frac{\epsilon'}{\widehat\epsilon^{2+2D}}\right).$$
\end{lemma}
\begin{proof}

Following~\cite[\S 4]{MasurSmillie91}, we define a \emph{Delaunay triangulation} of a translation surface $\omega \in \hh$.
Consider the Voronoi decomposition of the translation surface with respect to the flat metric: each zero of $\om$ determines a cell given by points which are closer to it than
to any other zero, and have a unique shortest geodesic connecting
the two. The dual decomposition is the Delaunay decomposition of the surface,
and any triangulation given by a further dissection of this is a \emph{Delaunay triangulation}
of $\om$. 

\paragraph*{\bf A basis of saddle connections} Using the  Delaunay triangulation,  for each $\omega$ we take a basis $\{\gamma_i\}_{i=1}^D$ for $H_1(X,\Sigma, \mathbb Z)$ of saddle connections, where the lengths are all bounded below by $\widehat{\epsilon}$.  
Here $\Sigma$ is the set of zeroes of $\omega$.  Let $z_i=x_i+iy_i = \int_{\gamma_i} \omega$ be the holonomy vectors of the basis.  The assumption of thickness says $|z_i|\geq \widehat\epsilon$.  It follows from~\cite[\S 4]{MasurSmillie91} that the lengths are also bounded above by $\frac{1}{\widehat{\epsilon}}$.
Since the thick part of the stratum is compact we can cover it by finitely many coordinates maps via the map $$\eta \mapsto \left(\int_{\gamma_1} \eta, \ldots \int_{\gamma_D} \eta\right)$$ which gives local coordinates on $\hh$ in a neighborhood of $\omega$.  For the remainder of the proof we will work in a single chart and to simplify notation we will use $(z_1,\ldots, z_D)$ as the coordinate charts. Our measure $\mu$ arises from Lebesgue measure in these period coordinates. 

\paragraph*{\bf Efficiency} There are a finite number of possible Delaunay triangulations (up to the action of the mapping class group, see, for example, Harer~\cite{Harer}), so it is enough to bound the measure  for each Delaunay triangulation. In \cite[\S 2]{AthreyaCheungMasur19} it was shown that Delaunay triangulations are \emph{efficient}. By this we mean that up to a fixed multiplicative constant, the  length of any saddle connection is bounded below by the length of a homologous path of edges in the triangulation joining the same endpoints. 
It follows that since $|z|,|w|\leq L$, we can write the holonomies $z,w$ in this basis as $\sum_{i=1}^D n_i z_i$ and $\sum_{i=1}^D m_i z_i$ respectively with coefficients $n_i,m_i \in\mathbb Z$ that are  $O\left(\frac{1}{\widehat\epsilon}\right)$, where the implied bound depends on $L$. We wish to compute the Lebesgue measure of the subset of $\mathbb{C}^D$   where we have one of the relations.   

\paragraph*{\bf Close vertical components} For the first where we have a relation  between $|\im z|$ and $|\im w|$ the assumption that the ratio of the imaginary parts of the holonomy  $z,w$ are within $\epsilon'$ of each other gives  for some $m_i,n_i,z_i$  
\begin{equation*} 
\left|\frac{\sum_{i=1}^D n_iy_i}{\sum_{i=1}^D m_i y_i}-1\right|<\epsilon'
\end{equation*}
which by the bound on $m_i,n_i$ and $y_i$ implies 
\begin{equation}\label{eq:hypplane(1)}
\left |\sum_{i=1}^D(n_i-m_i)y_i \right|=O\left(\frac{\epsilon'}{\widehat\epsilon^2}\right).
\end{equation} For each fixed $D$-tuple of integers $n_i-m_i$ the set of $(z_1,\ldots, z_D)$ which satisfy Equation~\ref{eq:hypplane(1)} forms a neighborhood of width $O(\frac{\epsilon'}{\widehat\epsilon^{2}})$ around the $\R^{2D-1}$ hyperplane formed by the  linear condition on the imaginary parts of $z_i = x_i + i y_i$. So the Lebesgue measure in $\C^D$ within this fixed chart is $O(\frac{\epsilon'}{\widehat\epsilon^{2}})$. Since $|n_i-m_i|= O(\frac{1}{\widehat\epsilon})$ there are  $O(\frac{1}{\widehat{\epsilon}^D})$ integer
tuples when taking the differences and so the total measure  is $O(\frac{\epsilon'}{\widehat\epsilon^{2+D}})$. 
 This gives a bound for the measure of the set of $\omega$ satisfying the first statement in the definition of $\Omega(\widehat{\epsilon}, \epsilon', L,L')$. 
 
\paragraph*{\bf Virtual area close to $A$} For the second we have $$\left|\left(\sum_{i=1}^D n_ix_i\right)\left(\sum_{i=1}^D m_iy_i\right)-\left(\sum_{i=1}^D m_ix_i\right)\left(\sum_{i=1}^D n_iy_i\right)-A\right|<\epsilon'.$$
 For fixed $D$ tuples  $(n_1,\ldots, n_D)$ and $(m_1,\ldots, m_D)$ the set of possible $z_i = x_i+ i y_i$ lives in the $\epsilon'$ neighborhood of a hyperplane determined by the determinant condition, and thus for fixed $A$ has Lebesgue measure $O(\epsilon')$. In this case there are $O(\frac{1}{\widehat{\epsilon}^{2D}})$ such tuples and so the result follows.
 
\paragraph*{\bf Cases 3 and 4} The proof of the measure bound  for $\omega$ satisfying the third and fourth  statements are similar.  Write $z=\sum n_iz_i$. The third assumption is $$\left|\sum_{i=1}^D n_iy_i- L'\right|<\epsilon'.$$  The measure of the set of $z$ that satisfy this inequality for some $n_i$ is  $O(\frac{\epsilon'}{\widehat\epsilon^{D}})$, which implies the desired bound since $\widehat \epsilon<1$. The proof of the last is similar.
  
\end{proof}

\subsection{Proof of Proposition~\ref{prop:Mjbounds}} \label{sec:finalproof}
In this section we will apply Lemma~\ref{lem:bad} and Corollary~\ref{cor:thin} to prove Proposition~\ref{prop:Mjbounds}.

\begin{proof}

We begin by proving the estimate (\ref{eq:mainbound})
For $(z,w) \in M_t$, we have, by Lemma~\ref{lem:equality}:
	\begin{align*} \left(\one_{D_A(e^t/2, e^t)} -(A_th_A)\right)(z,w) &= 1 - \pi e^{2t} \cdot \frac{\arctan(e^{-2t})}{\pi} \\ &= 1 - 1 + e^{2t} \cdot O(e^{-6t}) \\ &= O(e^{-4t})\\ \end{align*}
 Combining this with Proposition~\ref{prop:quadratic}, we have \begin{align}\label{eq:M0bound} |m_t(\om)| &\le O(e^{-4t})  \cdot N_A^*(\omega, e^t) \\ \nonumber &= O(e^{-4t})O(e^{2t}) \\ \nonumber &= O(e^{-2t}) = o(e^{2t}). \\ \nonumber \end{align}

\paragraph*{\bf Estimating error terms} We now begin to estimate the error terms.
Let $\epsilon'' >0$. We will get bounds of the form $O(\epsilon'' e^{2t})$. Since $\epsilon''$ is arbitrary this will lead to the desired bound $o(e^{2t})$. Choose $\widehat{\epsilon}$ so that $\epsilon'' = \widehat{\epsilon}^{\frac{1}{2N} - 2\delta}.$ Note that since $\delta < \frac{1}{4N}$, $\frac{1}{2N} - 2\delta>0$. Let $L = \sqrt{8}\pi\widehat{c}_A$, and choose  $\epsilon'$ small enough so that 
$$\frac{\epsilon'}{\widehat{\epsilon}^{2+2D}} \left(\frac{L}{\widehat{\epsilon}}\right)^{1+\delta} < \epsilon''.$$ Choose $T_0$ large enough so that whenever $t\geq T_0$, Corollary~\ref{cor:thin} holds. Corollary~\ref{cor:thin} shows that  the number of pairs (without restrictions to be in $E_t^j$) such that after rotating and flowing land us  in the $\widehat\epsilon$-thin part (where the shortest curve has length at most $\sigma^j<\widehat{\epsilon}$) is $O(\epsilon'' e^{2t})$. Since we have the desired bounds in the thin part, we now only need to consider the thick part of the stratum, so we suppose $|g_tr_{\theta_i}z_i| > \widehat{\epsilon}.$

\subsection*{Bounds on $z$ in trapezoid}\label{sec:trapezoid}
Suppose $\theta \in I(\theta_i)$ satisfies $r_{\theta}z \in g_{-t} \mathcal{T}$, where we first choose $t$ large enough so $$\tan(\pi e^{-2t}) \leq 6\pi e^{-2t}.$$ Notice that $\theta_z$, the angle which makes $r_{\theta_z}$ vertical has $r_{\theta_z}z, r_{\theta}z \in g_{-t}\mathcal{T}$, so $|\theta - \theta_z| \leq e^{-2t}$. Then for any $\theta' \in I(\theta_i)$, since $|\theta - \theta'| \leq 2\pi e^{-2t}$, we have 
\begin{equation}\label{eq:thetabounds}
|\theta' - \theta_z| \leq 2\pi e^{-2t} + e^{-2t} \leq 3\pi e^{-2t}.
\end{equation}
Notice $\im(r_{\theta'} z) \leq |r_{\theta_z}z|$, so we now want a lower bound for $\im(r_{\theta'}z)$. Indeed since 
$$|\re(r_{\theta'}z)| = \im(r_{\theta'z}) \tan(|\theta_z - \theta'|) \leq 
\im(r_{\theta'}z) \tan(3\pi e^{-2t}) \leq\im(r_{\theta'}z) 6\pi e^{-2t},$$
this implies
$$\im(r_{\theta'}z)^2(1 +36\pi^2 e^{-4t}) \geq\im(r_{\theta'}z)^2 + \re(r_{\theta'}z)^2 \geq |z|^2 = |\im(r_{\theta_z}z)|^2.$$
Thus for any $\theta' \in I(\theta_i)$, 
\begin{equation}
	\label{eq:zbounds}
	\frac{|\im(r_{\theta_z}z)|}{\sqrt{1+36\pi^2e^{-4t}}} \leq \im(r_{\theta'}z)  \leq |\im(r_{\theta_z}z)|
\end{equation}
\subsection*{Bounds on $w$}
We claim the imaginary parts of $r_{\theta_z}w$ and $r_{\theta'}w$ don't differ too much. To do this we will use polar coordinates, so set 
$$r_{\theta_z} w = |w|e^{i\psi_z}\quad\text{and} \quad r_{\theta'}w = |w| e^{i\psi'}.$$
Recall by Equation~\eqref{eq:thetabounds}, $|\psi'-\psi_z| = |\theta_z-\theta'| \leq 3\pi e^{-2t}$. Then since circumference $|\psi'-\psi_z|$ is bigger than the length of a chord on the unit circle,
$$3\pi e^{-2t} \geq |\theta' - \theta_z| = |\psi' - \psi_z| \geq |e^{i\psi'} - e^{i\psi_z}| \geq |\sin(\psi')-\sin(\psi_z)|.$$
Dividing by $\sin(\psi_z)$, we have
\begin{equation}\label{eq:wpsi}\left|\frac{\sin(\psi')}{\sin(\psi_z)} -1\right| \leq \frac{3\pi e^{-2t}}{|\sin(\psi_z)|}.\end{equation}

\subsection*{Error term $E_t^1$}\label{sec:errorproof1} 

Choose $T_1\geq T_0$ large enough so that for $t\geq T_1$, 
$$\max\left\{\left(\frac{1}{2} - \frac{1}{2\sqrt{1+36\pi^2 e^{-4t}}}\right), \left(\frac{\sqrt{1+e^{-4t}}}{2}- \frac{1}{2}\right) \right\} \leq \epsilon'.$$
Combining Equation~\eqref{eq:zbounds} with the fact that
$$\frac{e^t}{2} \leq \im(r_{\theta_z} z) = |z| \leq \frac{e^t}{2} \sqrt{1+e^{-4t}},$$
$$\frac{1}{2\sqrt{1+ 36\pi^2 e^{-4t}}} \leq \im(g_tr_{\theta'}z) \leq \frac{\sqrt{1+e^{-4t}}}{2}.$$
By our choice of $T_1$, 
$\frac{1}{2}-\epsilon' \leq \im(g_tr_{\theta'} z) \leq \frac{1}{2} + \epsilon'.$
Thus the resulting surface by Lemma~\ref{lem:expand} satisfies
$$g_tr_{\theta'}(X,\omega) \in \Omega\left(\frac{\widehat{\epsilon}}{\sqrt{8}\pi}, \epsilon', L, \frac{1}{2}\right).$$

\subsection*{Error term $E_t^2$.}\label{sec:errorproof2}

Choose $T_2\geq T_0$ large enough so that for $t\geq T_2$, 
$$\max\left.\begin{cases} \sqrt{1+36\pi^2e^{-4t}} \left(1+3\pi e^{-2t}\left({\frac{\sqrt{1+e^{-4t}}}{1 - 4Ae^{-2t} }}\right)\right)-1,\\
1-\left(1-3\pi e^{-2t}\left({\frac{\sqrt{1+e^{-4t}}}{1 - 4Ae^{-2t} }}\right)\right) \frac{1 - 4Ae^{-2t} }{\sqrt{1+e^{-4t}}}\end{cases}\right\} \leq \epsilon'.$$
Since $|z| = \im(r_{\theta_z}z)$, 
$$|\im(r_{\theta_z}z)| \frac{1}{\sqrt{1+e^{-4t}}} \leq |w| \leq |\im(r_{\theta_z}z)|.$$
The determinant condition guarantees small angle. That is if $\psi$ is the angle of $r_{\theta_z} w$ from the vertical, then
$$\frac{|\re(r_{\theta_z} w)|}{|w|} = |\sin(\psi)| \leq \frac{A}{|w|\cdot |z|} \leq \frac{A{\sqrt{1+e^{-4t}}}}{|z|^2} \leq \frac{4A e^{-2t}}{\sqrt{1+e^{-4t}}}.$$
 Thus
 $$\frac{|\im(r_{\theta_z}z)|}{{\sqrt{1+e^{-4t}}}}  \leq |\re(r_{\theta_z} w)| + |\im(r_{\theta_z} w)| \leq |\im(r_{\theta_z} w)| + \frac{4A e^{-2t}}{\sqrt{1+e^{-4t}}} |\im(r_{\theta_z} z)|,$$
 which yields
 \begin{equation}\label{eq:M2upper}
 \frac{|\im(r_{\theta_z}w)|}{\im(r_{\theta_z} z)} \geq \frac{1}{{\sqrt{1+e^{-4t}}}} - \frac{4Ae^{-2t}} {\sqrt{1+e^{-4t}}} = \frac{1 - 4Ae^{-2t} }{\sqrt{1+e^{-4t}}}.
 \end{equation}
 In the other direction we have the easier inequality that 
 \begin{equation}
 \label{eq:M2lower}\frac{|\im(r_{\theta_z}w)|}{\im(r_{\theta_z}z)} \leq \frac{|w|}{|z|} \leq 1.
 \end{equation}
Thus by our assumption on $|w|$, Equation~\eqref{eq:wpsi} and Equation~\eqref{eq:M2upper},
\begin{equation} \label{eq:wbounds}\left|\frac{\im(r_{\theta'}w)}{\im(r_{\theta_z} w)} -1\right| \leq 3\pi e^{-2t}\frac{|w|}{|\im(r_{\theta_z}w)|} \leq 3\pi e^{-2t}\frac{\im(r_{\theta_z}z)}{|\im(r_{\theta_z}w)|} \leq 3\pi e^{-2t}\left({\frac{\sqrt{1+e^{-4t}}}{1 - 4Ae^{-2t} }}\right).
\end{equation}
In the first direction we combine \eqref{eq:wbounds}, \eqref{eq:zbounds}, and \eqref{eq:M2lower} to obtain
\begin{align*}\frac{|\im(r_{\theta'}w)|}{\im(r_{\theta'}z)} \leq  \sqrt{1+36\pi^2e^{-4t}} \left(1+3\pi e^{-2t}\left({\frac{\sqrt{1+e^{-4t}}}{1 - 4Ae^{-2t} }}\right)\right) \leq 1+\epsilon'.
\end{align*}
In the other direction, combine \eqref{eq:wbounds}, \eqref{eq:zbounds}, and \eqref{eq:M2upper} to obtain
\begin{align*}
	\frac{|\im(r_{\theta'}w)|}{\im(r_{\theta'}z)} &\geq \left(1-3\pi e^{-2t}\left({\frac{\sqrt{1+e^{-4t}}}{1 - 4Ae^{-2t} }}\right)\right) \frac{1 - 4Ae^{-2t} }{\sqrt{1+e^{-4t}}}\geq 1-\epsilon'.
\end{align*}
Thus by Lemma~\ref{lem:expand}, and noting that flowing by $g_t$ does not change the ratio of  imaginary parts,
$$g_{t}r_{\theta'}\omega \in \Omega\left(\frac{\widehat{\epsilon}}{\sqrt{8}\pi}, \epsilon', L, 1\right).$$

\subsection*{Error term $E_t^3$}\label{sec:errorproof3}

Choose $T_3\geq T_0$ large enough so that for $t\geq T_3$,
 
$$\max\left\{\left(1 - \frac{1}{\sqrt{1+36\pi^2 e^{-4t}}}\right), \left(\sqrt{1+e^{-4t}}- 1\right) \right\} \leq \epsilon'.$$
For any $\theta' \in I(\theta_i)$, combining \ref{eq:zbounds} with the fact that in this case $$e^t \leq \im(r_{\theta_z} z)= |z| \leq e^t \sqrt{1+e^{-4t}},$$
we have
$$\frac{1}{\sqrt{1+ 36\pi^2 e^{-4t}}} \leq \im(g_tr_{\theta'}z) \leq \sqrt{1+e^{-4t}}.$$
By our choice of $T_3$, 
$1-\epsilon' \leq \im(g_tr_{\theta'} z) \leq 1+ \epsilon'.$
Thus the resulting surface by Lemma~\ref{lem:expand} satisfies
$$g_tr_{\theta'}\omega \in \Omega\left(\frac{\widehat{\epsilon}}{\sqrt{8}\pi}, \epsilon', L,1\right).$$

\subsection*{Error term $E_t^4$}\label{sec:errorproof4}

Choose $T_4\geq T_0$ large enough so that for $t\geq T_4$, 

$$\max\left.\begin{cases} \sqrt{1+\rho_Ae^{-4t}} \sqrt{1+36\pi^2e^{-4t}} \left(1+3\pi e^{-2t}\left(\frac{\sqrt{1+\rho_Ae^{-4t}}}{1-  4Ae^{-2t}\sqrt{1+\rho_Ae^{-4t}}} \right)\right)-1,\\
1- (1 -4Ae^{-2t}\sqrt{1+\rho_A e^{-4t}}) \left(1-3\pi e^{-2t}\left(\frac{\sqrt{1+\rho_Ae^{-4t}}}{1-  4Ae^{-2t}\sqrt{1+\rho_Ae^{-4t}}} \right)\right) \end{cases}\right\} \leq \epsilon'.$$
Since $|z| = \im(r_{\theta_z}z)$, 
$$|\im(r_{\theta_z}z)| \leq |w| \leq |\im(r_{\theta_z}z)| \sqrt{1+\rho_A e^{-4t}}.$$
The determinant condition guarantees small angle. That is if $\psi$ is the angle of $r_{\theta_z} w$ from the vertical, then
$$\frac{|\re(r_{\theta_z} w)|}{|w|} = |\sin(\psi)| \leq \frac{A}{|w|\cdot |z|} \leq \frac{A}{|z|^2} \leq 4A e^{-2t}.$$
 Thus
 $$|\im(r_{\theta_z}z)| \leq |w|\leq |\re(r_{\theta_z} w)| + |\im(r_{\theta_z} w)| \leq |\im(r_{\theta_z} w)| +4A e^{-2t} |\im(r_{\theta_z} z)| \sqrt{1+\rho_A e^{-4t}},$$
 which yields
 \begin{equation}\label{eq:M4upper}
 \frac{|\im(r_{\theta_z}w)|}{\im(r_{\theta_z} z)} \geq 1 -4Ae^{-2t}\sqrt{1+\rho_A e^{-4t}}.
 \end{equation}
 In the other direction we have the easier inequality that 
 \begin{equation}
 \label{eq:M4lower}\frac{|\im(r_{\theta_z}w)|}{\im(r_{\theta_z}z)} \leq \frac{|w|}{|z|} \leq \sqrt{1+\rho_Ae^{-4t}}.
 \end{equation}

\noindent Thus by our assumption on $|w|$, Equation~\eqref{eq:wpsi} and Equation~\eqref{eq:M4upper},
\begin{align}\label{eq:M4wbounds}\left|\frac{\im(r_{\theta'}w)}{\im(r_{\theta_z} w)} -1\right| \leq 3\pi e^{-2t}\frac{|w|}{|\im(r_{\theta_z}w)|} &\leq 3\pi e^{-2t}\frac{\im(r_{\theta_z}z)}{|\im(r_{\theta_z}w)|}\sqrt{1+\rho_Ae^{-4t}} \\ \nonumber &\leq 3\pi e^{-2t}\left(\frac{\sqrt{1+\rho_Ae^{-4t}}}{1 -4Ae^{-2t}\sqrt{1+\rho_A e^{-4t}}} \right).
\end{align}
Now in the first direction we combine Equations~\eqref{eq:zbounds},\eqref{eq:M4wbounds} and \eqref{eq:M4lower} to obtain
\begin{align*}\frac{|\im(r_{\theta'}w)|}{\im(r_{\theta'}z)} 
\leq\sqrt{1+\rho_Ae^{-4t}} \sqrt{1+36\pi^2e^{-4t}} \left(1+3\pi e^{-2t}\left(\frac{\sqrt{1+\rho_Ae^{-4t}}}{1-  4Ae^{-2t}\sqrt{1+\rho_Ae^{-4t}}} \right)\right) \leq 1+\epsilon'.
\end{align*}
In the other direction, combine Equations~\eqref{eq:M4wbounds}\eqref{eq:zbounds}, and \eqref{eq:M4upper} to obtain
\begin{align*}
	\frac{|\im(r_{\theta'}w)|}{\im(r_{\theta'}z)} &\geq (1 -4Ae^{-2t}\sqrt{1+\rho_A e^{-4t}}) \left(1-3\pi e^{-2t}\left(\frac{\sqrt{1+\rho_Ae^{-4t}}}{1-  4Ae^{-2t}\sqrt{1+\rho_Ae^{-4t}}} \right)\right) \geq 1-\epsilon'.
\end{align*}
Thus by Lemma~\ref{lem:expand}, and noting that flowing by $g_t$ does not change the ratio of imaginary parts,
$$g_{t}r_{\theta'}\omega \in \Omega\left(\frac{\widehat{\epsilon}}{\sqrt{8}\pi}, \epsilon', L,L'\right).$$

\subsection*{Combining cases}\label{sec:proofcombine}
	We conclude by Lemma~\ref{lem:bad}, that  
$$\mu\left( \Omega\left(\frac{\widehat{\epsilon}}{\sqrt{8}\pi}, \epsilon', L,L'\right)\right) = O\left(\frac{\epsilon'}{\widehat{\epsilon}^{2+2D}}\right).$$ Using the relative homology coordinates given by the  Delaunay triangulation we identify $\Omega\left(\frac{\widehat{\epsilon}}{\sqrt{8}\pi}, \epsilon', L,L'\right)$ with a domain in $\mathbb{C}^D$. Notice that all coordinates have length at least $\frac{\widehat{\epsilon}}{\sqrt{8}\pi}$. Let $h$  be the characteristic function of the compact set $\Omega\left(\frac{\widehat{\epsilon}}{\sqrt{8}\pi}, \epsilon', L,L'\right)$. Choose a small radially symmetric neighborhood $U$ of $\Omega\left(\frac{\widehat{\epsilon}}{\sqrt{8}\pi}, \epsilon', L,L'\right)$ such that all coordinates have absolute value at least    $\frac{\widehat{\epsilon}}{\sqrt{16}\pi}$  and choose  a continuous  $g\in C_0^\infty(\mathbb{C}^D)$ supported in $U$ such that  $h\leq g$, and 
$$\int_{\hh} g d\mu=O\left(\frac{\epsilon'}{\widehat{\epsilon}^{{2+2D}}}\right).$$   Let $H$  be a product of $D$ annuli  $A(r_0,r_1)$ that contains $\Omega$  with $r_0>0$. 

\paragraph*{\bf A family of functions} We  consider the  family of functions  $\mathcal{F}$ as in Lemma~\ref{lem:separate} in this case with $N=D$  and $H$ as above.  Let $\phi$ be a radially symmetric continuous function of compact support which is identically $1$ on  $\Omega\left(\frac{\widehat{\epsilon}}{\sqrt{8}\pi}, \epsilon', L,L'\right)$.  and consider the family $\bar{\mathcal{F}}$=$\{\phi f:f\in\mathcal{F}\}$.  

This is a $K$-finite compactly supported family and so there is $\bar f\in\bar{\mathcal{F}}$ uniformly close to $g$. Since we only need upper bounds, we can choose $g$ larger if necessary so that we can  choose a smoothing function $\eta$ with $$\int_{-\infty}^\infty \eta(u)du=1$$ and support close enough to $0$ so that $ h\leq (\eta*\bar f)$.  By Theorem~\ref{thm:Nevo}, 
$$\lim_{t\to\infty}\int_{-\infty}^\infty \eta(t-s) (A_s \bar f)(\om)ds = \int_{\hh} \bar f d\mu = O\left(\frac{\epsilon'}{\widehat{\epsilon}^{2+2D}}\right).$$
which  for $t$ large enough gives $$A_t h(\om)=O\left(\frac{\epsilon'}{\widehat{\epsilon}^{2+2D}}\right).$$
Then as in Equation~\eqref{int:bound}, by Lemma~\ref{lem:expand}, for $t$ sufficiently large
$$\#\left\{I(\theta_i): \exists \theta \in I(\theta_i) \text{ so that } r_{\theta}z \in g_{-t}\mathcal{T} \text{ for } z\in E_t^k\right\} = O\left(\frac{\epsilon'}{\widehat{\epsilon}^{2+2D} }e^{2t}\right).$$
Moreover by Lemma~\ref{lem:EM} for each $I(\theta_i)$ the number of possible $z \in E_t^k$ is $O\left(\left[\frac{L}{\widehat{\epsilon}}\right]^{1+\delta}\right).$
Thus by our choice of $\epsilon'$, as desired
$$|E_t^k| = O\left(\frac{\epsilon'}{\widehat{\epsilon}^{2+2D} }e^{2t} \left[\frac{L}{\widehat{\epsilon}}\right]^{1+\delta}\right) = O(\epsilon'' e^{2t}).$$
\noindent Finally we remark that since $\epsilon''$ is arbitrary, the above estimate is $o(e^{2t})$. 
\end{proof}

\subsection{Proof of Lemma~\ref{lem:smallboundary}}\label{sec:integralbounds}
In order to prove Lemma~\ref{lem:smallboundary}, we first prove an integral bound for a function supported in the thin part of the stratum. For $\omega\in \mathcal{H}_\epsilon$ let $$\psi_{L_0, \epsilon}(\omega) = \chi_{B(0, L_0)^2}^{\SV}(\omega) \chi_{\hh_{\epsilon}}(\omega).$$ 
$\psi_{L_0, \epsilon}(\omega)$ counts the number of pairs of saddle connections of length at most $L_0$ if $\omega \in \hh_{\epsilon}$, and if $\omega\notin \mathcal{H}_\epsilon$, $\psi_{L_0, \epsilon}(\omega)=0$.  For ease of notation we write $\psi = \psi_{L_0, \epsilon}$.
\begin{lemma}
\label{lem:integralthin}
$$\int_{\mathcal{H}} \psi d\mu  =O(\epsilon^{\frac{1}{N}-2\delta}),$$ where the implied bound depends on $L_0$. 
\end{lemma}
\begin{proof}

Again by Lemma~\ref{lem:MS},
$$\mu\left(F\left(j\right)\right) = O\left(\sigma^{2j + 2\frac{j}{2N}}\right) = O\left(\sigma^{2j} \sigma^{\frac{j}{N}}\right).$$

\noindent Applying Lemma~\ref{lem:EM},  which says for each $\omega\in F(j)$ we have $\psi(\omega)=O(\sigma^{-j(2+2\delta)})$, gives   
$$\int_{F(j)} \psi d\mu =O(\sigma^{j(\frac{1}{N}-2\delta)}).$$

\noindent By Lemma~\ref{lem:ACM}, for each $\omega\in  G(j)$ we have $\psi(\omega)=O\left(\left(\sigma^{j/2N}\right)^{-2N}\right)=O(\sigma^{-j})$.  We also have by Lemma~\ref{lem:MS} that $\mu (G(j))=O(\sigma^{2j})$ giving 
$$\int_{G(j)} \psi d\mu=O(\sigma^j).$$ 

\noindent Then again if $\omega \in H(j)$ $\psi(\om) =O(\sigma^{-j(2+2\delta)})$ and now 
$$\mu(H(j)) = O(\sigma^{3j}),$$
where the implied constant depends on $\epsilon_0$.
This gives $$\int_{H(j)} \psi d\mu =O(\sigma^{j(1-2\delta)}).$$ \noindent and so 
$$\int_{\mathcal{H}_\epsilon} \psi d\mu =O\left (\sum_{j\geq j_0} \sigma^j+\sigma^{j(1/N-2\delta)}+\sigma^{j(1-2\delta)}\right )=O\left(\sigma^{j_0(1/N-2\delta)}\right)=O\left(\epsilon^{1/N-2\delta}\right).$$

\end{proof}

\begin{proof}[Proof of Lemma~\ref{lem:smallboundary}]

For any $\epsilon'$, consider the set $$U \subset \{(z,w) \in \C^2: 1/4 \leq  \im(z) \leq 2, |\re(z)| \leq \im(z)+1, \mbox{ and }|w| \leq 3\sqrt{1 + (8A + 16 A^2)}\}$$ where at least one of the following conditions holds:
\begin{enumerate}
\item  Area close to $A$:
$$A(1-\epsilon') \leq |z\wedge w|\leq (1+\epsilon')A$$
\item Ratio of imaginary parts close: $$(1-\epsilon')\im z
\leq |\im w|\leq (1+\epsilon')\im z$$
\item Real part and imaginary part of $z$ close: $$(1-\epsilon')\im z\leq |\re z|\leq (1+\epsilon')\im z$$
\item Imaginary part of $z$ close to $1$: $$(1-\epsilon') \leq |\im(z)| \leq 1+\epsilon'$$ 
\item  Imaginary part of $z$ close to $1/2$: $$(\frac{1}{2}-\epsilon') \leq |\im(z)| \leq \frac{1}{2}+\epsilon'.$$ 
\end{enumerate}

\paragraph*{\bf A neighborhood of $\partial R_A(\mathcal T)$} The set $U$ describes a neighborhood of   $\partial R_A(\mathcal{T} )$. It can be chosen small enough as not to contain any points of the form $(z,0)$ or $(0,w)$. We can then find an $r_0>0$  and enlarge $U$ so that is  of the form $A(r_0,r_1)\times A(r_0,r_1)$ for $0<r_0<r_1$. This is possible as the boundary of the trapezoid is bounded uniformly away from zero, and  $\partial\mathcal{R}_A(z)$ is  bounded away from $0$ by some fixed distance $r(z)>0$, and since $z$ varies in a compact set, $$r_0 = \min_{z\in \mathcal{T}} r(z) > 0.$$

\noindent Choose  $g=g_{\epsilon}$ continuous such that 
\begin{enumerate}

\item $g(z,w)\leq 1$ and is supported in $R_A(\mathcal{T})\cup U$
\item $g(z,w)=1$ for  $(z,w)\in R_A(\mathcal{T})$
\end{enumerate}

\paragraph*{\bf A bounding function} Let $\phi=\phi_\epsilon'$ be continuous, supported on $U$, such that $\phi\leq 1$,  and $\phi=g$ on $U\setminus R_A(\mathcal{T})$.  (The point is that $\phi=1$ on $\partial R_A(\mathcal{T})$). 
Then $$(g-h_A) \leq \phi.$$

\paragraph*{\bf Radial symmetry} Since $U$ is radially symmetric, we can apply
 Proposition~\ref{prop:convergescont}. We conclude
$$\lim_{\tau\to\infty}(A_\tau \widehat\phi)(\om)ds = \int_{\hh} \widehat\phi d\mu=\int_{\hh_{\widehat\epsilon}}\widehat\phi d\mu+\int_{\hh\setminus\hh_{\widehat\epsilon}}\widehat\phi d\mu,$$ where $\hh_{\widehat\epsilon}$ is the $\widehat\epsilon$-thin part.  
By Lemma~\ref{lem:integralthin} the first term on the right is $$O\left(\widehat \epsilon^{\frac{1}{N}-2\delta}\right).$$ By Lemma~\ref{lem:EM}, each $\omega$ in the  $\widehat\epsilon$ thick part has $O\left(\frac{1}{\widehat\epsilon^{2+2\delta}}\right)$ pairs of saddle connections of bounded length.  This together with   Lemma~\ref{lem:bad} says the second term is $$O\left(\frac{\epsilon'}{\widehat{\epsilon}^{4+2D+2\delta}}\right).$$ These two inequalities  imply that for  $\tau$ large enough,   $$A_\tau( \widehat g-\widehat h_A)(\om)=O(\widehat\epsilon^{\frac{1}{N}-2\delta})+O\left(\frac{\epsilon'}{\widehat\epsilon^{4+2D+2\delta}}\right).$$

\paragraph*{\bf Conclusion} Recall we fixed $\delta$ so that $\frac{1}{N}>4\delta$. Then given  $\epsilon$,  choose
$\widehat\epsilon$ so the first term is at most $\epsilon/2$.
Then choose $\epsilon'$ so the second term is also at most $\epsilon/2$.
The first conclusion of the Lemma follows.
The second conclusion follows directly from Lemma~\ref{lem:bad} and Lemma~\ref{lem:integralthin}.

\end{proof}

\bibliographystyle{alpha}
\bibliography{Sources}
\end{document}